\documentclass[12pt]{article}
\usepackage{bbm} 
\usepackage{amsmath,amssymb,amsthm,amsfonts}
\usepackage{latexsym, comment}
\usepackage{tikz, xcolor}

\numberwithin{equation}{section}

\begin{document}

\newtheorem{prop}{Proposition}[section]
\newtheorem{cor}{Corollary}[section] 
\newtheorem{theo}{Theorem}[section]
\newtheorem{lem}{Lemma}[section]
\newtheorem{rem}{Remark}[section]
\newtheorem{con}{Conjecture}[section]
\newtheorem{as}{Assumption}[section]
\newtheorem{de}{Definition}[section]
\newtheorem{notation}{Notations}[section]

\newtheorem*{theo*}{Theorem}

\newcommand\cA{\mathcal A}
\newcommand\cB{\mathcal B}
\newcommand\cC{\mathcal C}
\newcommand\cD{\mathcal D}
\newcommand\cE{\mathcal E}
\newcommand\cF{\mathcal F}
\newcommand\cG{\mathcal G}
\newcommand\cH{\mathcal H}
\newcommand\cI{\mathcal I}
\newcommand\cJ{\mathcal J}
\newcommand\cK{\mathcal K}
\newcommand\cL{{\mathcal L}}
\newcommand\cM{\mathcal M}
\newcommand\cN{\mathcal N}
\newcommand\cO{\mathcal O}
\newcommand\cP{\mathcal P}
\newcommand\cQ{\mathcal Q}
\newcommand\cR{{\mathcal R}}
\newcommand\cS{{\mathcal S}}
\newcommand\cT{{\mathcal T}}
\newcommand\cU{{\mathcal U}}
\newcommand\cV{\mathcal V}
\newcommand\cW{\mathcal W}
\newcommand\cX{{\mathcal X}}
\newcommand\cY{{\mathcal Y}}
\newcommand\cZ{{\mathcal Z}}

\newcommand\Z{\mathbb Z}
\newcommand\N{\mathbb N}
\newcommand\E{\operatorname{\mathbb E{}}}
\renewcommand\P{\operatorname{\mathbb P{}}}
\newcommand\Var{\operatorname{Var}}
\newcommand\Cov{\operatorname{Cov}}
\newcommand\Corr{\operatorname{Corr}}
\newcommand\Exp{\operatorname{Exp}}
\newcommand\Po{\operatorname{Po}}
\newcommand\Bi{\operatorname{Bi}}
\newcommand\Bin{\operatorname{Bin}}
\newcommand\Mul{\operatorname{Mul}}
\newcommand\Be{\operatorname{Be}}
\newcommand\Ge{\operatorname{Ge}}
\newcommand\NBi{\operatorname{NegBin}}
\newcommand\Res{\operatorname{Res}}
\newcommand\fall[1]{^{\underline{#1}}}
\newcommand\rise[1]{^{\overline{#1}}}
\newcommand\whp{w.h.p.}

\renewcommand{\=}{:=}

\newcommand\ga{\alpha}
\newcommand\gb{\beta}
\newcommand\gd{\delta}
\newcommand\gD{\Delta}
\newcommand\gf{\varphi}
\newcommand\gam{\gamma}
\newcommand\gG{\Gamma}
\newcommand\gk{\varkappa}
\newcommand\gl{\lambda}
\newcommand\gL{\Lambda}
\newcommand\go{\omega}
\newcommand\gO{\Omega}
\newcommand\gs{\sigma}
\newcommand\gss{\sigma^2}
\newcommand\gth{\theta}
\newcommand\gthx{\vartheta}
\newcommand\eps{\varepsilon}
\newcommand\ep{\varepsilon}

\newcommand\tc{t\cx}
\newcommand\cx{_\mathsf{c}}
\newcommand\tcx{t\cxx}
\newcommand\acx{a\cxx}
\newcommand\cxx{_\mathsf{c}^*}
\newcommand\tx{t_*}
\newcommand\ac{a\cx}
\newcommand\Ax{A^*}

\newcommand\ntoo{\ensuremath{{n\to\infty}}}
\newcommand{\tend}{\longrightarrow}
\newcommand\dto{\overset{\mathrm{d}}{\tend}}
\newcommand\pto{\overset{\mathrm{p}}{\tend}}
\newcommand\oi{[0,1]}
\newcommand\Op{O_{\mathrm p}}
\newcommand\op{o_{\mathrm p}}

\newcommand\cao{\cA(0)}
\newcommand\bc{b\cx}
\newcommand\bcx{b^*}
\newcommand\gnp{\ensuremath{G_{n,p}}}

\newcommand\cDA{\Delta\cA}

\newcommand\set[1]{\ensuremath{\{#1\}}}

\newcommand\setn{\set{1,\dots,n}}

\newcommand\ettfir{1-\gf(\ga)^{r-1}}

\newcommand\simin{\in}

\newcommand\AsN{\mathrm{AsN}}

\newcommand\tpi{{\tilde \pi}}

\newcommand\qq{^{1/2}}
\newcommand\qw{^{-1}}
\newcommand\qwr{^{-1/r}}
\newcommand\qww{^{-2}}

\newcommand\bigpar[1]{\bigl(#1\bigr)}
\newcommand\Bigpar[1]{\Bigl(#1\Bigr)}
\newcommand\etto{\bigpar{1+o(1)}}
\newcommand\ettop{\bigpar{1+o_p(1)}}

\newcommand\floor[1]{\lfloor#1\rfloor}

\newcommand\dd{\,\textup{d}}
\newcommand\ddq{\textup{d}}
\newcommand\ddt{\frac{\ddq}{\ddq t}}

\newcommand\lrpar[1]{\left(#1\right)}
\newcommand\parfrac[2]{\lrpar{\frac{#1}{#2}}}

\newcommand\punkt[1]{\if.#1\else.\spacefactor1000\fi{#1}}
\newcommand\ie{i.e\punkt}

\newcommand\PTRF{\emph{Probab. Theory Rel. Fields} }

\newcommand\ind{{\mathbbm 1}}

\newenvironment{romenumerate}[1][0pt]{% optional argument changes indentation
\addtolength{\leftmargini}{#1}\begin{enumerate}% gives (i), (ii) etc.
 \renewcommand{\labelenumi}{\textup{(\roman{enumi})}}%
 \renewcommand{\theenumi}{\textup{(\roman{enumi})}}%
 }{\end{enumerate}}

\setcounter{page}{1}
\renewcommand{\theequation}{\thesection.\arabic{equation}}

%%%%%%%%%%%%%%%%%%%%%%%%%%%%%%%%%%%%%%%%%%%%%%%%%%%%%%%%%%%%%%%%%%%%%%%%%%%%%%%%%%%%%%%%%%%%%%%%%%%%%%%%%%%%%%%%%%%%%

\title{Bootstrap percolation \\
  on a graph with random and local
  connections. }

\author{
Tatyana S. Turova 
\thanks{ Mathematical Centre, University of
Lund, Box 118, Lund S-221 00, 
Sweden.}
\and Thomas Vallier
\thanks{
Department of Mathematics and Statistics, Box 68, FI-00014 University of Helsinki, Finland
 }}

\date{}
\maketitle

\begin{abstract}
Let  $G_{n,p}^1$ be a
superposition of the random graph $G_{n,p}$ and a one-dimensional
lattice: the $n$ vertices are set to be
on a ring with fixed  edges between the consecutive vertices, 
and with random independent edges given with probability $p$ between any pair of vertices.
Bootstrap percolation on a random graph 
is a process of spread of 
"activation'' 
on a given realization of the graph with a given number of
initially active nodes. At each step those vertices which have not been
active but have at least $r \geq 2$ active neighbours become active as well.  
We study the size of the final active set in the limit when $n\rightarrow
\infty $. The parameters of the model
are  $n$, the size $A_0=A_0(n)$ of the initially active set and the
probability $p=p(n)$ of the edges in the graph. 

Bootstrap percolation process on $G_{n,p}$ was studied earlier. 
Here we show that the addition of  $n$ local connections to the graph 
$G_{n,p}$ 
leads to a more narrow
critical window for the phase transition, preserving however, 
the critical scaling of parameters known for the model on $G_{n,p}$.
 We discover a range of  parameters which 
 yields percolation on $G_{n,p}^1$  but not on $G_{n,p}$.

\end{abstract}

{\it MSC2010 subject classifications.} 05C80, 60K35, 60C05.


\section{Introduction}
Bootstrap percolation was introduced on a Bethe lattice by Chalupa, Leath and Reich \cite{CLR}
 to model some magnetic systems. Also, models of 
neuronal activity have very similar basic features. 
(Use of percolation models in neuronal sciences was predicted already by Harris \cite{Ha}.)  
Therefore we   define here a bootstrap percolation 
on a  graph $G$ 
 as the spread of
\emph{activation} 
in the following way.
Assume $G$ has a finite set of vertices, call it $V$. 
There is an initial set $\cA_0 \subset V$ of
\emph{active} vertices.
For a given threshold $r \ge 2$, each inactive vertex that has at least $r$ 
active neighbours (i.e., the vertices connected to it by the graph edges)
becomes active and can spread the activation along its edges. 
This is repeated until no more vertex becomes active.
Active vertices never become inactive, so the set of active vertices
grows monotonously. Let $\cA^*$ denote the final active set.
 We say that (a sequence of)
$\cA_0$ \emph{percolates} (completely) if $|\cA^*| = \Ax=n$ and that $\cA_0$  \emph{almost percolates} if the number of
vertices that remain inactive is $o(n)$, \ie, if $\Ax=n-o(n)$.

Bootstrap percolation  has been 
extensively studied on varieties of graphs, as e.g., 
 $d$-dimensional grid (see recent results by
Balogh, Bollob\'as, Duminil-Copin and Morris
\cite{BBDM} and Uzzell \cite{U}), hypercube (Balogh and  Bollob{\'a}s \cite{BB}),
infinite trees
(Balogh, Peres and Pete \cite{BPP}), random regular graphs (Balogh and Pittel \cite{BP},
Janson \cite{SJ215}), 
Erd\H os--R\'enyi random graph $\gnp$   (Janson, {\L}uczak, Turova and Vallier \cite{JLTV}),  Galton-Watson trees
(Bollob\'as,  Gunderson,  Holmgren,  Janson, and 
Przykucki \cite{BG}).

We study here a graph with both local and global links. 
This is a simplification of the  model  introduced in  \cite{TV} as a 
model of neuronal activity
(see e.g., also \cite{Tl}, \cite{T},  \cite{Ko}, \cite{KP} for the 
related studies). It is known that in a neuronal tissue the synaptic 
connections between neurons form a very complex network, 
where the strength of the connections may 
depend on the physical distances, as well as may be modelled  as "random",
associating the probability of a connection with its strength. Hence, considering two types of the connections is a step towards more complex model. 
(Notice a difference with "the small world network''
of Newman and Watts \cite{NW}:  we do not re-wind  edges as in 
\cite{NW}, but we consider a superposition of a lattice and a random graph on the same set of vertices.) 

The process of bootstrap percolation models the  propagation of impulses in the neuronal network: roughly speaking, in order to be activated a 
neuron should get a large enough number of incoming impulses. This is the main feature  of the bootstrap percolation: a vertex is activated if it is connected to a certain (but strictly greater than 1) number of active vertices.

Despite a long history of the subject and even very detailed results both for the $d$-dimensional grid and for  random graph (see the citations above), still
only a few theoretical results are available for  graphs with mixed connections. 
Recently a new percolation process, the so-called jigsaw percolation
was introduced and studied in \cite{BCDS}, and then developed and further investigated  in \cite{GS}. This process is indeed closely related to the one treated here, it
may also  evolve on a graph where the deterministic geometry of a lattice is combined with random independent connections between the 
vertices. Notice, that in the model of  jigsaw percolation the random edges (typically as in the Erd{\H o}s-R{\'e}nyi graph) represent some social links or ideas
 ("people graph''), while the other structure, as a lattice, for example, represents objects, or "puzzles". Therefore  edges from these two graphs 
in the definition of the jigsaw percolation
play distinct roles, unlike in our model (motivated by problems in neuroscience).  
In particular,  the mechanism of merging clusters in the models considered in 
\cite{BCDS} and \cite{GS} is different: roughly speaking, jigsaw percolation
is faster than bootstrap percolation.

We start with a regular one dimensional lattice. The vertices $V=\{1,...,n\}$ are ordered on a ring $R_n$ and have an edge with their two nearest neighbours. 
We add random connections. The random edges are given independently for each pair of
 vertices  with the same probability $p$. Hence, there might be at most two edges between the
vertices in the model, and if there are two edges between a pair of
vertices, the edges are necessarily of two types: one from the random
graph and another one from the lattice. In such a case, we merge the edges. The subgraph on $V$ with the random edges only is a random graph $G_{n,p}$. 
Similarly, replacing a ring by a $d$-dimensional torus with $n$ vertices
one can define a graph $G^d_{n,p}$ for all $d>1$.
One can also study bootstrap percolation on a 1-dimensional lattice where a vertex has a link with the vertices at distance at most $k$. 

We consider a bootstrap percolation on $G_{n,p}^1$ with the threshold
$r=2$ and $p=p(n)$.
We assume, that 
an initial set $\cA(0)$ consists of a given number $A_0 = A_0(n)\geq 2$ of vertices
chosen uniformly at random from the set $\{1, \ldots , n\}$.  
We study here the process with  the threshold $r=2$ for simplicity, but also for the fact that $2$ is a "critical'' value for the percolation on $\Z$ where each vertex has at most 2 connections. However, it should be possible to extend the results for the case with $r>2$ and possibly more local deterministic edges on  $\Z$.

Typically, a bootstrap percolation process exhibits a 
 threshold phenomenon: either
 $o(n)$ number of vertices become active,
or, on the contrary,   $n-o(n)$ vertices become active. 
The main question here is how the superposition of different structures affects the phase transition. In particular, is it possible to get a complete percolation combining two subcritical systems? 
In the case of an ordinary percolation model,  a superposition of
two subcritical graphs (one being a grid with randomly removed edges, bond percolation, and another one being an Erd\H os-R\'enyi random graph,
 each of which has the largest connected component of order at most $\log n$)  may have a component of order $n$  \cite{VT}. 
In this case superposition of the graphs produces new critical values in a phase diagram (see \cite{VT}).
We shall see here that  the bootstrap percolation process exhibits different properties (at least in dimension $1$).

\section{Results}

Let us recall some notations and results from \cite{JLTV} which we
need here. 
\subsection{Notations}
Let $1\leq i < j \leq n$, the distance between the vertices $i$ and $j$ is defined as
\begin{equation*}\label{eq:dist}
d(i,j) = \min \left\{ j-i , n+i-j \right\}.
\end{equation*}
The distance of a vertex $u$ to a set $ \cS$ is defined as
\begin{equation*}\label{eq:distset}
d(v,\cS) = \inf \left\{ d(u,v), v \in \cS \right\}.
\end{equation*}
We denote $\partial_1 \cS$ the outer boundary of a vertex set $\cS$ on $R_n$:
\begin{equation*}\label{eq:outer}
\partial_1 \cS = \left\{ v \in R_n\setminus \cS , d(v,\cS) =1 \right\} .
\end{equation*} 

We use the notations $O_{L^k}$ and $o_{L^k}$, as well as $O_{P}$ and  $o_P$, for the random variables in the same setting as in \cite{SJN6}. For example, 
let $a_n$ be some sequence of real numbers, then 
$X_n = O_{L^k} (a_n) \Leftrightarrow \E \left(|X_n|^k\right) = O\left((a_n)^k\right)$.
In particular, $X_n = O_{L^2} (a_n) \Rightarrow X_n = O_{L^1} (a_n) \Rightarrow X_n = O_{P} (a_n)$.

We use the notation $f(n)= \Theta\left(g(n)\right)$ as $c_1 g(n) \leq f(n) \leq c_2 g(n)$ for $c_1, c_2 >0$ and as $n \to \infty$.
We write that an event holds with high
probability (\whp{}) if the probability of this event tends to 1 as \ntoo.
Note that, for example,  `$=o(1)$ \whp{}' is equivalent to `$=o_P(1)$'
and to `$\pto0$'.

All unspecified limits are as \ntoo.

For given $n$ and $p$ define
\begin{equation*}\label{eq:ac}
 \ac\=\frac{1}{2} \tc = \frac{1}{2} \frac{1}{np^2}.
\end{equation*}
The term $\ac$ is the first-order term of the critical threshold 
$\acx (n,p)$ for bootstrap percolation on the random graph $\gnp$. The term $\acx (n,p)$ is defined in \cite{JLTV} as follows. Let
\begin{equation*}  \label{tpi}
\tpi(t)\=\P\bigpar{\Po(tp)\ge 2}
=\sum_{j=2}^\infty\frac{(pt)^j}{j!}e^{-pt},
\end{equation*}
where $\Po(tp)$ denotes a Poisson random variable with mean $tp$. Then set 
\begin{equation*}\label{eq:acx}
\acx\=-\min_{t\le3\tc}\frac{n\tpi(t)-t}{1-\tpi(t)},
\end{equation*}
and let $\tcx\in[0,3\tc]$ be the point where the minimum is attained.
Notice that $\tc = \frac{1}{np^2}$ is also the first-order term of $\tcx$.

\subsection{Results}

Let here
${\cal A}^*_0$ denote the final set of 
vertices activated due to a bootstrap percolation on a
random graph $G_{n,p}$ starting with $A_0$ active
vertices, i.e., we do not take into
account the local edges from the ring $R_n$.
It is clear that there is a coupling of these two models
(with and without the short edges) such that
\begin{equation}\label{co}
{\cal A}^*_0 \subseteq {\cal A}^*.
\end{equation}

The following theorem
(which is a  particular case when $r=2$ of the theorem proved in
\cite{JLTV} for a  general case $r\geq 2$.)
describes the phase transitions in the value $|{\cal A}^*_0|$ depending on the
initial condition $A_0$. 
Let $b_c =  pn^2 e^{-pn}$, which  is the expected number of vertices of degree 1 in $\gnp$.

\begin{theo*}[Theorem 3.1 \cite{JLTV}, case $r=2$] 
Suppose that  $n\qw\ll p\ll n^{-1/2}$.
Let $A^*_0=|{\cal A}^*_0|$ be the total number of
vertices activated due to a bootstrap percolation on a
random graph $G_{n,p}$ starting with $A_0$ active vertices. 
  \begin{romenumerate}
\item  \label{T3.1i}
  If\/ $A_0/\ac\to \ga<1$, then
  $$\Ax_0=(\gf(\ga)+o_P(1))\tc,$$ where
$\gf(\ga)=1-\sqrt{1-\ga}$ with $\lim_{\ga \to 1} \gf (\ga) = \gf (1) =1$. 
\item \label{T3.1ii}
  If\/ $A_0/\ac\ge 1+\gd$, for some $\gd>0$, then
  $\Ax_0=n-O_P(b_c)=n-o_P(n)$;
in other words, we have \whp{} almost percolation. 
  \end{romenumerate}
\end{theo*}

Due to the observation (\ref{co}), if the initial set $\cA_0$ percolates
 on $G_{n,p}$,
then the same set  percolates on  $G_{n,p}^1$ 
which contains
$G_{n,p}$.
Therefore, we
 are interested here in the initial conditions  $\cA_0$ 
which do
not yield a  percolation on $G_{n,p}$, i.e., when $\alpha \leq  1$ in the
above theorem. The following theorem tells us that 
adding 
to  graph $G_{n,p}$ the edges between the nearest neighbours (in dimension one)
does not change much the subcritical regime,
at least when 
 $p\geq \frac{\log n}{n}$. 

\begin{theo}\label{T1}
Suppose that  $n^{-1} \ll  p\ll n^{-1/2}$.
Let $A^*=|{\cal A}^*|$ be the total number of
vertices activated due to a bootstrap percolation on a
random graph $G_{n,p}^1$ starting with $A_0$ active vertices. which are chosen uniformly at random from the vertex set $\{1,\ldots, n\}$.
  \begin{romenumerate}
\item\label{T1i}
  If 
$p\gg  \frac{\log n}{n \log (pn)}$ 
and 
$A_0/\ac\to \ga<1$, then
  $$\Ax=(\gf(\ga)+\op(1))\tc.$$ 
\item \label{T1ii}
  If \/ $A_0/\ac\ge 1+\gd$, for some $\gd>0$, then \whp{} 
  $\Ax=n-o(n)$;

if also $p\gg \frac{1}{2}\frac{\log n+\log(pn)}{n} $
then w.h.p. $\Ax=n$, i.e.,  we have w.h.p. complete percolation.
\end{romenumerate}
\end{theo}

\begin{rem}\label{RM1}
The condition $p\gg  \frac{\log n}{n \log (pn)}$  in  Theorem \ref{T1} is satisfied, e.g., if $p\geq \frac{\log n}{n}$.
\end{rem}

Theorem \ref{T1}  does not describe the case when $p$ is close to 
$\frac{1}{n}$. Notice that 
 for  $p$ of order  $1/n$, 
addition of $n$ edges changes the graph properties.
What follows from our proof is that the subcritical phase 
for very small $p$ may have a large  
number of steps   before 
the process stops. 

It turns out that it is the critical case, i.e.,  when
$\alpha =1$, which is affected most by the presence 
of the local connections. First we recall the situation with $G_{n,p}$.

\begin{theo*}[Theorem 3.6  \cite{JLTV}]\label{Tac}
Suppose that $n\qw\ll p\ll n^{-1/2}$. Let $A^*_0$ be the total number of
vertices activated due to a bootstrap percolation (with threshold
$r=2$) on a random graph $G_{n,p}$ starting with $A_0$ active vertices.
  \begin{romenumerate}
\item
  If\/ $(A_0-\acx)/\sqrt{\ac}\to -\infty$, then
for every $\eps>0$, \whp{}  $A^*_0\le \tcx\le\tc(1+\eps)$.
If further $A_0/\acx\to1$, then $A^*_0=(1+\op(1))\tc$.
\item
  If\/ $(A_0-\acx)/\sqrt{\ac}\to +\infty$, then
$A^*_0=n-\Op(b_c)$.
\item
  If\/ $(A_0-\acx)/\sqrt{\ac}\to y\in(-\infty,\infty)$, then
for every $\eps>0$ and every $\bcx\gg\bc$ with $\bcx=o(n)$,
\begin{align*}
  \P(A^*_0>n-\bcx) &\to \Phi\bigpar{y},
\\
  \P\bigpar{A^*_0\in[(1-\eps)\tc,(1+\eps)\tc]} &\to 1-\Phi\bigpar{ y}.
\end{align*}
  \end{romenumerate}
\end{theo*}

In the following, 
we show that when 
$p$ is small enough, including short edges into the model may lead to percolation even when there is no percolation in $G_{n,p}$
with the same parameters. 

\begin{theo}\label{T2}
 Let $A^*$ be the total number of
vertices activated due to a bootstrap percolation on a random graph $G_{n,p}^1$ starting with $A_0$ active vertices, chosen uniformly out of the vertex set $\{1, \ldots, n\}.$ Assume, $A_0/\acx\to1$.
  \begin{romenumerate}
\item \label{T2i}
  If 
$n^{-1}\ll p \ll n^{-3/4}$ and either $A_0>\acx$,  or, in the case $A_0\leq \acx$,
\begin{equation}\label{sp}
 \frac{\acx -A_0}{\sqrt{\ac} } =   o\left(  \frac{1}{pn^{3/4}}\right)^2,
\end{equation}
then \whp{}  $A^* = n -o(n)$.
\item \label{T2ii}
 If $n^{-2/3}\ll p \ll n^{-1/2}$ and 
 $$
 \frac{A_0- \acx}{\sqrt{\ac} } 
\to -\infty,$$ then
for every $\eps>0$, \whp{}  $A^* \le \tcx\le\tc(1+\eps)$.
  \end{romenumerate}
\end{theo}

Theorem \ref{T2} part \ref{T2i}
 describes the case when the addition of local edges even in dimension 1 changes the phase diagram.
Indeed, condition (\ref{sp}) tells us that almost percolation happens not only 
whenever
 $A_0 \geq \acx$ but even when $A_0 < \acx$, if  $A_0$  deviates  from 
$\acx$ at most by $\sqrt{\ac}o\left( \frac{1}{pn^{3/4}}\right)^2$ which  under the
assumption $n^{-1}\ll p \ll n^{-3/4}$ 
may be much larger than $\sqrt{\ac}$. If the latter occurs, then by Theorem 3.6 (i) \cite{JLTV} cited above under the same conditions percolation will not occur on the edges of 
$G_{n,p}$ only. 
Part \ref{T2ii} tells us that the critical window does not change for "large'' $p$, i.e., when  $n^{-2/3}\ll p \ll n^{-1/2}$. 
(Observe that it does not lead to a contradiction, since $a_c= a_c(p)$ 
changes accordingly.)

Theorem \ref{T2} does not cover the case $n^{-3/4}\leq p \leq n^{-2/3}$. 
Our analysis suggests, however, that almost percolation will happen 
even when $n^{-3/4}\leq p \leq n^{-2/3}$
 under the same condition (\ref{sp}) . Notice that the right side of (\ref{sp}) is bounded for $p\geq n^{-3/4}$ and moreover, it is $o(1)$ if  $p\gg n^{-3/4}$.
Hence, one may think that in all other cases but Part \ref{T2i}  the critical deviation $\acx-A_0 $  for the percolation is of order  $\sqrt{\ac}$.

\section{Discussion on higher dimensions}

We show that adding the structure of the one-dimensional grid 
 makes  the phase transition even sharper by decreasing the critical window. 

The challenge remains to study a bootstrap percolation process on 
$G_{n,p}^d$ with $d>1$.  In this case the effect of the local connections from the $d$-dimensional grid will be substantial, as one can readily see in the following calculations. Consider for simplicity a two-dimensional discrete torus $T=[1, \ldots ,N]^2$ with $n=N^2$ vertices and all edges between these vertices inherited from the two-dimensional lattice. Assume also that with a probability $p$ there is an edge between any pair of vertices, independent for different pairs. Denote the corresponding graph $G_{n,p}^2$. Assume that with probability $q=q(n)$
each  vertex is set initially to  be active 
independently of the rest and consider a bootstrap percolation with threshold $r=2$ as in \cite{H}.  
It is known  (see Holroyd \cite{H},  and Balogh, Bollob\'as, Duminil-Copin and Morris
\cite{BBDM} for the latest development in the area) that 
a complete percolation on torus $T$ 
with local edges only,
will happen w.h.p. 
if $q(n)/q_c(n,2,2)>1$, where
\begin{equation*}\label{crT}
 q_c(n,2,2):=\frac{\pi^2}{18 \log n} (1+o(1))=:\frac{c_0}{ \log n} (1+o(1)).
\end{equation*}
Otherwise, if $q(n)/q_c(n,2,2)<1$, the complete percolation  w.h.p. will not 
occur. 

Consider now a bootstrap percolation process on 
$G_{n,p}^2$
with 
$$A_0= \alpha  n \frac{c_0}{\log n} =: \alpha a_c = \ga n q_c \etto $$ 
initially active vertices, and with 
 \[p=\frac{1}{\sqrt{2na_c}}. \]
Let $0<\alpha<1$, and therefore
\[ \frac{A_0}{nq_c(n,2,2)} = \frac{\alpha}{1+o(1)}<1.\]
Using results \cite{H} one concludes that  on 
the subgraph $T = [1, \ldots ,N ]^2$ of $G_{n,p}^2$ induced by the  local connections only, 
  a complete percolation will not occur with probability tending to $1$ as $n\to \infty$ (or equivalently as $N \to \infty$). 
Also, by the Theorem 3.1 \cite{JLTV} (see above)
on 
the subgraph of $G_{n,p}$  of 
 $G_{n,p}^2$  bootstrap percolation process with a high probability
ends with only
$$ A^*_0 = (1-\sqrt{1-\alpha}) 2a_c=o(n)$$ 
active vertices. 
Hence, neither short edges nor random edges alone may yield with a high probability a complete percolation  on $G_{n,p}^2$ with the given parameters. 
However, one can choose $0<\alpha<1$ so that 
$$\frac{A^*_0}{np_c(n,2,2)}= \frac{2( 1-\sqrt{1-\alpha})}{1+o(1)}>1.$$
Then starting 
with $A^*_0$ vertices
one can argue using again results \cite{H} that 
 a complete percolation will happen with a high probability on the graph
$G_{n,p}^2$. This confirms that a superposition of two subcritical systems can lead to almost percolation. 

Besides these heuristics a complete analysis of 
bootstrap percolation on a graph with mixed edges in dimension greater than 1 remains to be an open problem.

\section{Proofs}\label{sec:model}

\subsection{Useful reformulation}

We shall distinguish the following three types of activation 
 of a vertex depending on the type of connections which
caused this activation:
 the  long range activation, the  short range activation and the mixed activation. The long range activation uses only random edges, 
we also call it "$\gnp$ activation".
The short range activation uses only local edges: if the vertices $i-1$ and $i+1$ are active then  the vertex $i$ becomes active as well. 
We say that the activation of a vertex is mixed if it is caused by one edge of each type.
See figure 1.

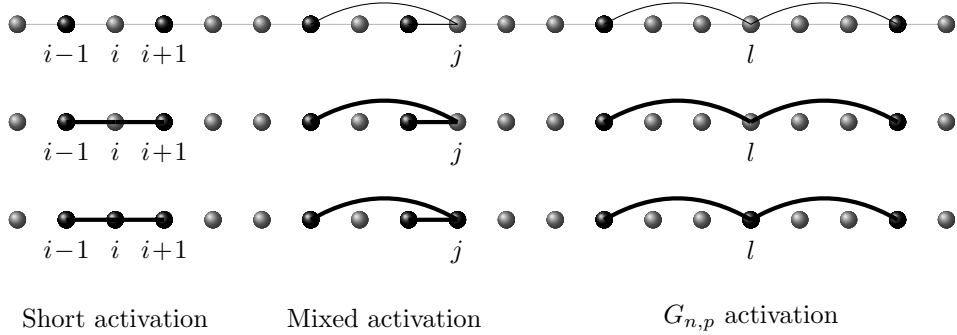
\begin{figure}[h]
\begin{center}
\begin{tikzpicture}[scale=1.3, every to/.style={bend right}]
\draw[color=gray, very thin, opacity=0.6](0.5,0)--(10,0);
\foreach \nodes in {0.5,1,...,10}
\shade[ball color=gray, opacity=0.6](\nodes,0) circle(2.5pt);

% percolation activation

\shade[ball color=black](1,0) circle(2.5pt);
%\shade[ball color=black](1.5,0) circle(2.5pt);
\shade[ball color=black](2,0) circle(2.5pt);

%% nodes
\draw (1,-0.1) node[below] {\footnotesize{$i \! -\! 1$}};
\draw (1.5,-0.1) node[below] {\footnotesize{$i$}};
\draw (2,-0.1) node[below] {\footnotesize{$i \! + \! 1$}};

% mixed activation
\shade[ball color=black](3.5,0) circle(2.5pt);
\shade[ball color=black](4.5,0) circle(2.5pt);
%\shade[ball color=black](5,0) circle(2.5pt);

\draw[-](5,0) to (3.5,0);
\draw[-](4.5,0)--(5,0);

%% nodes
\draw (5,-0.1) node[below] {\footnotesize{$j$}};

% gnp activation
\shade[ball color=black](6.5,0) circle(2.5pt);
%\shade[ball color=black](8,0) circle(2.5pt);
\shade[ball color=black](9.5,0) circle(2.5pt);

\draw[-](8,0) to (6.5,0);
\draw[-](9.5,0) to (8,0);

%% nodes
\draw (8,-0.1) node[below] {\footnotesize{$l$}};

% 2nd SCOPE

\begin{scope}[yshift=-1cm]

\foreach \nodes in {0.5,1,...,10}
\shade[ball color=gray, opacity=0.6](\nodes,0) circle(2.5pt);
% percolation activation
\shade[ball color=black](1,0) circle(2.5pt);
%\shade[ball color=black](1.5,0) circle(2.5pt);
\shade[ball color=black](2,0) circle(2.5pt);

\draw[ultra thick](1,0)--(1.5,0);
\draw[ultra thick](2,0)--(1.5,0);

%% nodes
\draw (1,-0.1) node[below] {\footnotesize{$i \! -\! 1$}};
\draw (1.5,-0.1) node[below] {\footnotesize{$i$}};
\draw (2,-0.1) node[below] {\footnotesize{$i \! + \! 1$}};

% mixed activation
\shade[ball color=black](3.5,0) circle(2.5pt);
\shade[ball color=black](4.5,0) circle(2.5pt);
%\shade[ball color=black](5,0) circle(2.5pt);

\draw[ultra thick](5,0) to (3.5,0);
\draw[ultra thick](4.5,0)--(5,0);

%% nodes
\draw (5,-0.1) node[below] {\footnotesize{$j$}};

% gnp activation
\shade[ball color=black](6.5,0) circle(2.5pt);
%\shade[ball color=black](8,0) circle(2.5pt);
\shade[ball color=black](9.5,0) circle(2.5pt);

\draw[ultra thick](8,0) to (6.5,0);
\draw[ultra thick](9.5,0) to (8,0);

%% nodes
\draw (8,-0.1) node[below] {\footnotesize{$l$}};

\end{scope}

% 3rd SCOPE 

\begin{scope}[yshift=-2cm]

\foreach \nodes in {0.5,1,...,10}
\shade[ball color=gray, opacity=0.6](\nodes,0) circle(2.5pt);
% percolation activation
\shade[ball color=black](1,0) circle(2.5pt);
\shade[ball color=black](1.5,0) circle(2.5pt);
\shade[ball color=black](2,0) circle(2.5pt);

\draw[ultra thick](1,0)--(1.5,0);
\draw[ultra thick](2,0)--(1.5,0);

%% nodes
\draw (1,-0.1) node[below] {\footnotesize{$i \! -\! 1$}};
\draw (1.5,-0.1) node[below] {\footnotesize{$i$}};
\draw (2,-0.1) node[below] {\footnotesize{$i \! + \! 1$}};

\draw (1.5, -1) node{\footnotesize{Short activation} };

% mixed activation
\shade[ball color=black](3.5,0) circle(2.5pt);
\shade[ball color=black](4.5,0) circle(2.5pt);
\shade[ball color=black](5,0) circle(2.5pt);

\draw[ultra thick](5,0) to (3.5,0);
\draw[ultra thick](4.5,0)--(5,0);

%% nodes
\draw (5,-0.1) node[below] {\footnotesize{$j$}};

\draw (4.25, -1) node{\footnotesize{Mixed activation} };

% gnp activation
\shade[ball color=black](6.5,0) circle(2.5pt);
\shade[ball color=black](8,0) circle(2.5pt);
\shade[ball color=black](9.5,0) circle(2.5pt);

\draw[ultra thick](8,0) to (6.5,0);
\draw[ultra thick](9.5,0) to (8,0);

%% nodes
\draw (8,-0.1) node[below] {\footnotesize{$l$}};

\draw (8, -1) node{\footnotesize{$\gnp$ activation} };

\end{scope}

\end{tikzpicture}
\end{center}
\caption{The $3$ different types of activation.}\label{graph:gnpact}
\end{figure}
In order to Analise the bootstrap percolation process on
$G_{n,p}^1$, we split the process of activation
in two distinct phases depending on the type of activation.

\subsubsection{First Exploration Phase.}\label{FEP}
Consider activation through the long (random)
connections only.

We say that the vertices are {\it neighbours} if there is at least one
edge between them. For any subgraph $G$ of $G_{n,p}^1$, we shall say
that two vertices are {\it $G$-neighbours} if there is an edge from
the subgraph $G$ between them.

We follow the algorithm for revealing the activated vertices as described in
\cite{JLTV}.
First, we change the time scale: we consider at each time step the activations from one vertex only. 

Given a set $\cA_0$ define $\cA_1(0)=\cA_0$. Choose $u_1 \in \cA_1(0)$ and give each of its neighbours a \emph{mark};
we then say
that $u_1$ is \emph{used}, and let
$\cZ_1(1)\=\set{u_1}$ be the set of used vertices
at time 1.

We continue recursively. At time $t>1$, choose (again uniformly at random)
a vertex
$u_{t}\in\cA_1(t-1)\setminus\cZ_1(t-1)$.
We give each $G_{n,p}$-neighbour of $u_{t}$ a new mark. 
Denote $M_v(t)$, the number of marks of the vertex $v \in V \setminus \cA(0)$ at time $t$, and let
 $\cS_1(t)$ be the set of  vertices outside of $\cA_1(t-1)$ with at least $2$ marks at time $t$: $\cS_1(t)= \left\{ v \notin \cA(0): M_v(t) \geq 2 \right\}$. 

Let us introduce Bernoulli random variables $\xi_{uv}\in Be(p)$ 
naturally associated with the edges of the random graph $G_{n,p}$: 
$\xi_{uv}=1$ if there is an edge between $u$ and $v$ in $G_{n,p}$, otherwise, $\xi_{uv}=0$.
Notice that $\xi_{u_iv}$  is also the indicator function that $v$ receives a mark at time $i$.
Denote $M_v(t)$, the number of marks of the vertex $v \in V \setminus \cA(0)$ at time $t$.
Then 
\begin{equation*}\label{cS}
\cS_1(t)= \left\{ v \notin \cA(0): M_v(t) \geq 2 \right\} =  \left\{v \not\in \cA (0): \sum_{i=1}^{t}\xi_{u_iv}\geq 2 \right\}.
\end{equation*}
Observe that the vertices of set $\cS_1(t)$ (more precisely, the labels of those vertices) are distributed uniformly over the set 
$\{1, \ldots, n\}$ (drawing  $|\cS_1(t)|$ points without replacement).
Using the independence of the connections on $\gnp$, we derive
\begin{equation*}\label{Sbin}
|\cS_1(t)| \stackrel{d}{=} \Bin (n-A_0,\pi_1(t)), 
\end{equation*}
where
\begin{equation}\label{p2}
\pi_1(t) = \P \left\{ v \in \cA_1 (t) \right\} = \P \left\{ M_v(t) \geq 2 \right\} = \P \left\{ \sum_{i=1}^t \xi_{u_i,v} \geq 2 \right\} = \P \left\{ \Bin (t,p) \geq 2 \right\}.
\end{equation}

Define now the set of  active vertices at time $t >0$ by
\begin{equation}\label{Act1}
\cA_1(t)= \cA_0 \cup \cS_1(t). 
\end{equation}
Finally, we let $\cZ_1(t)=\cZ_1(t-1)\cup\{u_{t}\}=\{u_s:s\le t\}$ be the set of used vertices.

The process stops when 
$\cA_1(t)\setminus\cZ_1(t)=\emptyset$, i.e.,  when 
all active vertices are used. We denote this time by $T_1$; 
\begin{equation}\label{eq:T1}
T_1 = \min \{t\ge0: \cA_1(t) \setminus \cZ_1(t)=\emptyset\}=
\min \{t\ge0: |\cA_1(t)|=t \}.
\end{equation}
 We call this phase an "exploration'' phase as we explore the long range connections of the vertices. The total number of active vertices at the end of this phase is denoted $|\cA_1(T_1 )| = T_1 $. 

\subsubsection{First Expansion Phase.}

 Now we take into account the structure of the local connections. Let us denote 
$R_n$ the corresponding subgraph of $G_{n,p}^1$ (which forms a Hamiltonian cycle on $V$).

 After the 1-st exploration phase we have a random set $\cA_1(T_1)$ of
 active vertices on $R_n$.  Hence, we may represent the set of inactive
 vertices  as a collection of paths on $R_n$. (A path on $R_n$ has a
 structure inherited from $R_n$: the consecutive vertices are pairwise
 connected.)

During the  "expansion'' phase, the set of active vertices $\cA_1 (T_1)$
may expand to its neighbours, or, in other words the paths of inactive
 vertices may become only shorter. More precisely, we define
the
 expansion phase in
 3
 different steps.
 
\begin{enumerate}
\item Any   vertex 
 which has two active (i.e., belonging to the set ${\cal A}_1(T_1)$) neighbours on $R_n$ becomes active.
This means that all the paths of inactive
 vertices which consist of a single vertex become active. 

After this step, we are left with the paths of inactive
 vertices which contain at least two vertices.
 Each of these vertices
  may have at most one mark assigned during the
 exploration phase.

 \item Any vertex (in any inactive path of length at least two)
 which  has a mark
 and which is
either an endpoint or is
 connected to an endpoint 
only through vertices each of which also has a mark, becomes
active.
  
\item  After the second
 step there may be again
paths of inactive
 vertices which contain a single vertex. Then step 1 is repeated,
 i.e., again any   vertex 
 which  has two active neighbours on $R_n$ becomes active.
\end{enumerate}
The third step completes  the expansion phase.

\begin{figure}[h]
\begin{center}
\begin{tikzpicture}[scale=1.3, every to/.style={bend left}] %, every ot/.style={bend left}]
\draw[color=gray, very thin, opacity=0.6](-1,0)--(8,0);
\foreach \nodes in {0.5,1,...,8}
\shade[ball color=gray, opacity=0.6](\nodes,0) circle(2.5pt);
% black nodes
\shade[ball color=black](0.5,0) circle(2.5pt);
\shade[ball color=black](2,0) circle(2.5pt);
\shade[ball color=black](4.5,0) circle(2.5pt);

\shade[ball color=black](6,0) circle(2.5pt);
\shade[ball color=black](7,0) circle(2.5pt);

\shade[ball color=black](-1,0) circle(2.5pt);
\shade[ball color=gray, opacity=0.6](-0.5,0) circle(2.5pt);
\shade[ball color=gray, opacity=0.6](0,0) circle(2.5pt);

% links

\draw[-](-1,0) to (1,0);
\draw[-](-1,0) to (1.5,0);

\draw[-](-1,0) to (2.5,0);
\draw(-1,0) to (3,0);

\draw[-](4,0) to (7,0);
%\draw[thick](6,1) ot (5,0);

% comments step
\begin{comment}
\draw (1.5,-0.2) node[below] {\footnotesize{step 2}};

\draw (3.5,-0.2) node[below] {\footnotesize{step 2}};
\draw (5,-0.2) node[below] {\footnotesize{step 2}};

\draw[->,thick] (4.5,1)--(4.5,0.3);
\draw (4.5,1.1) node[above]{\footnotesize{step 3}};

\draw (7.5,-0.2) node[below] {\footnotesize{step 1}};
\end{comment}

% 2nd SCOPE 

\begin{scope}[yshift=-1.2cm]

\draw[color=gray, very thin, opacity=0.6](-1,0)--(8,0);
\foreach \nodes in {0.5,1,...,8}
\shade[ball color=gray, opacity=0.6](\nodes,0) circle(2.5pt);

\shade[ball color=black](0.5,0) circle(2.5pt);
\shade[ball color=black](2,0) circle(2.5pt);
\shade[ball color=black](4.5,0) circle(2.5pt);

\shade[ball color=black](6,0) circle(2.5pt);
\shade[ball color=black](6.5,0) circle(2.5pt);
\shade[ball color=black](7,0) circle(2.5pt);

\shade[ball color=black](-1,0) circle(2.5pt);
\shade[ball color=gray, opacity=0.6](-0.5,0) circle(2.5pt);
\shade[ball color=gray, opacity=0.6](0,0) circle(2.5pt);

% links
\draw[ultra thick](6,0)--(6.5,0);
\draw[ultra thick](7,0)--(6.5,0);

\draw[-](-1,0) to (1,0);
\draw[-](-1,0) to (1.5,0);

\draw[-](-1,0) to (2.5,0);
\draw(-1,0) to (3,0);

\draw[-](4,0) to (7,0);

\draw (8.5,0) node[right] {\footnotesize{step 1}};
\end{scope}

% 3rd SCOPE 

\begin{scope}[yshift=-2.4cm]

\draw[color=gray, very thin, opacity=0.6](-1,0)--(8,0);
\foreach \nodes in {0.5,1,...,8}
\shade[ball color=gray, opacity=0.6](\nodes,0) circle(2.5pt);

\shade[ball color=black](0.5,0) circle(2.5pt);
\shade[ball color=black](1,0) circle(2.5pt);
\shade[ball color=black](1.5,0) circle(2.5pt);
\shade[ball color=black](2,0) circle(2.5pt);
\shade[ball color=black](2.5,0) circle(2.5pt);
\shade[ball color=black](3,0) circle(2.5pt);

\shade[ball color=black](4,0) circle(2.5pt);
\shade[ball color=black](4.5,0) circle(2.5pt);

\shade[ball color=black](6,0) circle(2.5pt);
\shade[ball color=black](6.5,0) circle(2.5pt);
\shade[ball color=black](7,0) circle(2.5pt);

\shade[ball color=black](-1,0) circle(2.5pt);
\shade[ball color=gray, opacity=0.6](-0.5,0) circle(2.5pt);
\shade[ball color=gray, opacity=0.6](0,0) circle(2.5pt);

% links
%\draw[ultra thick](7,0)--(7.5,0);
%\draw[ultra thick](8,0)--(7.5,0);

\draw[ultra thick](-1,0) to (1,0);
\draw[ultra thick](-1,0) to (1.5,0);

\draw[ultra thick](-1,0) to (2.5,0);
\draw[ultra thick](-1,0) to (3,0);

\draw[ultra thick](4,0) to (7,0);

\draw (8.5,0) node[right] {\footnotesize{step 2}};
\end{scope}

% 4th SCOPE 

\begin{scope}[yshift=-3.6cm]

\draw[color=gray, very thin, opacity=0.6](-1,0)--(8,0);
\foreach \nodes in {0.5,1,...,8}
\shade[ball color=gray, opacity=0.6](\nodes,0) circle(2.5pt);

\shade[ball color=black](0.5,0) circle(2.5pt);
\shade[ball color=black](1,0) circle(2.5pt);
\shade[ball color=black](1.5,0) circle(2.5pt);
\shade[ball color=black](2,0) circle(2.5pt);
\shade[ball color=black](2.5,0) circle(2.5pt);
\shade[ball color=black](3,0) circle(2.5pt);
\shade[ball color=black](3.5,0) circle(2.5pt);
\shade[ball color=black](4,0) circle(2.5pt);
\shade[ball color=black](4.5,0) circle(2.5pt);

\shade[ball color=black](6,0) circle(2.5pt);
\shade[ball color=black](6.5,0) circle(2.5pt);
\shade[ball color=black](7,0) circle(2.5pt);

\shade[ball color=black](-1,0) circle(2.5pt);
\shade[ball color=gray, opacity=0.6](-0.5,0) circle(2.5pt);
\shade[ball color=gray, opacity=0.6](0,0) circle(2.5pt);

% links
%\draw[ultra thick](7,0)--(7.5,0);
%\draw[ultra thick](8,0)--(7.5,0);

\draw[-](-1,0) to (1,0);
\draw[-](-1,0) to (1.5,0);

\draw[-](-1,0) to (2.5,0);
\draw(-1,0) to (3,0);

\draw[-](4,0) to (7,0);

\draw[ultra thick](3,0)--(3.5,0);
\draw[ultra thick](4,0)--(3.5,0);

\draw (8.5,0) node[right] {\footnotesize{step 3}};
\end{scope}

\end{tikzpicture}
\end{center}
\caption{The 3 different steps.}
\end{figure}
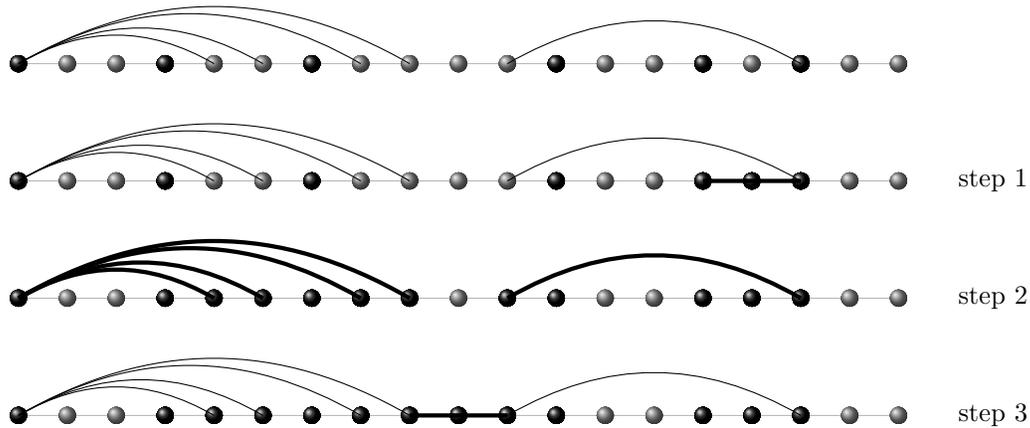

After the expansion phase, we may represent 
the set of inactive vertices
 as a collection of paths on $R_n$ each of which
 has the following properties:
\begin{romenumerate}
\item any path has at least two vertices,
\item the endpoints do not bear  a mark  but
 all the other vertices of the intervals may have at most one mark 
(assigned during the
 exploration phase).
\end{romenumerate}
Let us denote $\cD_1$ the set of vertices activated during the $1$-st
expansion phase.

At the end of the first expansion phase, 
we have $T_1+ |\cD_1|$ active vertices: $T_1$ of them have been used
and the set $\cD_1$ is still unused. 

\subsubsection{Alternating the phases.}

Having completed the 1st expansion phase, we shall alternate exploration  and expansion phases.
We shall denote ${\cal A}_{k}(T_k)$ and $\cD_k$ the sets of vertices acquired in the $k$-th exploration and expansion phases, correspondingly, $k\geq 1$. Notice  that  the sets ${\cal A}_{k}(T_k)$ and $\cD_k$ are disjoint.

We assume that after the $k$-th exploration phase we have used all vertices in 
$ {\cal A}_{k}(T_k)$ so that  $|{\cal A}_{k}(T_k)|=T_k$.
Let 
\[{\cal A}^k :=  \cup_{i=1}^{k} {\cal A}_{i}(T_i),\]
which is the set of all used vertices.
Still we have the set $\cD_k$ (assuming $\cD_k$ is not empty) of active vertices to explore: the ones which were
activated during the $k$-{th} expansion
phase.

Given the sets ${\cal A}_{i}(T_i), i\leq k,$ and $\cD_k$, 
let us  define the $k+1$-st exploration phase similar to the first one: we 
restart the process, setting again  time $t=0$,
but now on vertices $V \setminus {\cal A}^k$, 
among which the set of initially active vertices is
\[{\cal A}_{k+1}(0):= \cD_k.\]
This set  plays the same role as ${\cal A}_1(0)$ in the description
of the first exploration phase. 
Notice also that 
\[|V \setminus {\cal A}^k|= n -
\sum_{i=1}^{k}T_i.\]

We explore the vertices (i.e., assign
marks to their $G_{n,p}$-neighbours)  of $\cD_{k}$ one
at a time,  calling them again $u_1, u_2, \ldots $ . Observe, however,
that some of the vertices may have one mark from set ${\cal A}^k$ and
this makes the difference with the first exploration phase. More
precisely,  we have two
types of vertices: the vertices on the  boundary of set $ {\cal A}^k
\cup \cD_k$ which 
do not have any
random  edge to ${\cal A}^k$, and  the rest of  vertices (i.e., the ones
in $V \setminus ({\cal A}^k \cup
\cD_k \cup {\partial}_1({\cal A}^k \cup \cD_k)) $ which may have at most one
random  edge  to the set ${\cal A}^k$ (i.e., have a mark). 
Recall that $\partial_1 \left( \cA^k \cup \cD_k\right)$ is the outer boundary of $ \cA^k \cup \cD_k$, see \eqref{eq:outer}.
Then the set of vertices
activated during the first  $t$ steps of the $k+1$-st
exploration phase  is 
\begin{align*}\label{Skb}
 \cS_{k+1}(t):= & \left\{v \in {\partial}_1 ( {\cal A}^k \cup \cD_k): 
\sum_{i=1}^t\xi_{u_i v}\geq 2
 \right\}
\\ & \cup  \left\{v  \not\in {\cal A}^k \cup \cD_k \cup {\partial}_1(\cD_k
   \cup {\cal A}^k): \sum_{i=1}^t\xi_{u_i v} + \xi_v(k)  \geq 2 \right\},
   \nonumber
\end{align*}
where $\xi_v(k) \stackrel{d}{=} \xi(k)$ is an independent Bernoulli random variable which equals 
 one with the probability that an inactive
vertex $v$ has precisely one mark after the $k$-th exploration phase, i.e.,
\begin{equation}\label{pixi}
 \P\{ \xi_v(k)=1\}= \P\left\{ \Bin (|{\cal A}^k|,p) =1 \big| \Bin (|{\cal A}^k|,p)<2\right\}=
\frac{|{\cal A}^k|p}{1+(|{\cal A}^k|-1)p}.
\end{equation}
Let us define now
\begin{equation*}\label{pik}
 \pi_{k+1}(t)=\P\{\Bin (t,p) + \xi(k) \geq 2\},
\end{equation*}
where $\xi(k)$ and the binomial random variable are independent.
Notice that 
\begin{align*}\label{pik1}
 \pi_{k+1}(t) & =\P\{\Bin (t,p) \geq 2\} +\P\{\Bin (t,p)=1\}\P\{
 \xi(k)=1\}
 \\ &
=\pi_{1}(t)+
 \frac{|{\cal A}^k|p}{1+(|{\cal
     A}^k|-1)p}(1-p)^{t-1}pt,
 \nonumber
\end{align*}
where $\pi_{1}(t)$ is defined by  (\ref{p2}).
Then the distribution of $S_{k+1}(t):=|\cS_{k+1}(t)|$ is
\begin{equation}\label{Sbin*}
S_{k+1}(t)\stackrel{d}{=} \Bin  \left(n-|{\cal A}^k|-|\cD_k|-|{\partial}_1(\cD_k
   \cup {\cal A}^k)|,\pi_{k+1}(t)\right) + \Bin  \left(|{\partial}_1(\cD_k
   \cup {\cal A}^k)|,\pi_{1}(t)\right),
\end{equation}
where the binomial variables are independent. 
Define also (as in (\ref{Act1})) for $t>0$
\[{\cal A}_{k+1}(t)= \cD_k \cup \cS_{k+1}(t),  \]
which is the set of active vertices at the step $t$  of the $k+1$-st exploration phase.
Then, assuming $\cD_k \neq \emptyset$, the moment
\begin{equation*}\label{Tk}
T_{k+1}:= \min \{ t > 0 : |\cA_{k+1}(t)|  = t \} 
\end{equation*}
 is the first time when all the available active vertices are explored,
i.e., we have found all the $G_{n,p}$-neighbours of active vertices. 
This completes the $k+1$-st exploration phase. 

The $k+1$-st expansion phase is similar to the first one.
Recall that after the $k+1$-st  exploration phase we
 may
 represent the set of all remaining inactive vertices
 as a collection of intervals on $R_n$. Each of the vertices
 of these intervals may have at most one mark (assigned during any of the
 previous exploration phases). Then  at the $k+1$-st expansion phase 
 any vertex 
 which either has two active $R_n$-neighbours, or it has a mark and it 
 is connected to an endpoint with a mark 
through the vertices each of which has also a mark, becomes
active. Finish the phase with step 3 by activating the vertices that have two active nearest neighbours on $R_n$.
We denote $\cD_{k+1}$ the set of all vertices activated during this phase.

Let us now define  the process of bootstrap percolation on $\gnp^1$ 
as
\begin{equation*}\label{BP}
{\cal A}(t)= \cup_{i=1}^{k-1}{\cal A}_{i}(T_i) \cup {\cal
  A}_{k}\left(t-\sum_{i=1}^{k-1} T_i \right), \ \ \ \ \
\sum_{i=1}^{k-1} T_i \leq t < \sum_{i=1}^{k} T_i, \ \ \ \ \ k\geq 1.
\end{equation*}
The process of bootstrap percolation on $\gnp^1$ 
stops at time $T$ which is
\begin{equation}\label{ST}
T=\min\{t: |{\cal A}(t)| =t\}.
\end{equation}
It follows then that  
\begin{equation}\label{ST2}
T=\sum_{k=1}^K T_k ,
\end{equation}
 where
\begin{equation}\label{ST1}
K=\min\{k: \cD_k = \emptyset \},
\end{equation}
meaning that no vertex is  activated during the $k$-{th} expansion phase.
We shall denote
\[{\cal A}^*:={\cal A}(T).\]
Notice that by (\ref{ST}) and (\ref{ST2}) we have 
\begin{equation*}\label{Afi}
|{\cal A}^*|=\sum_{k=1}^K T_k .
\end{equation*}

\begin{rem}\label{RA}
By changing the time and considering the activation in different order,
we do not change the limiting set ${\cal A}^*$ of activated vertices
which depends only on 
the initial set ${\cal A}$. 
\end{rem}

%\vspace{0.5cm}
%%%%%%%%%%%%%%%%%%%%%%%%%%%%%%%%
\subsection{The number of vertices activated in an expansion phase.}

We begin with  
the
first expansion phase, namely, we shall study the set $\cD_1$.

\begin{lem}\label{lem1}
Let  $\cA_1(T_1)$
be a set of vertices uniformly distributed on $V=\{1,\ldots, n\}$,
and assume that
 $|\cA_1(T_1)| = T_1 \leq \frac{2}{np^2}$, where $n^{-1}\ll p \ll n^{-1/2}$. Then 
\begin{equation*}\label{eq:lem1}
 |\cD_1|  =  \left\{
\begin{array}{ll}
2pT_1^2
 +O_{L^1}\left(  \frac{pT_1^2}{pn} + \sqrt{p T_1^2}\right) , & \mbox{ if } p T_1^2 \rightarrow \infty, \\ \\
O_{L^1}\left(    pT_1^2 \right), & \mbox{ otherwise}. \\
\end{array}
\right.
\end{equation*}
\end{lem}

\begin{rem}
Notice that the random variable $T_1=A_0^*$ is described by Theorems 3.1 \cite{JLTV} and 3.6 \cite{JLTV} cited above.
\end{rem}
\begin{proof}[Proof of Lemma \ref{lem1}]
  For simplicity of the notations let us set here $T_1=k$. Given a subset $\cA_1(T_1) =\{i_1, \ldots ,i_k\}$ (assume that $i_1<
\ldots <i_k$)
define sets (maybe empty)
\begin{equation*}\label{I}
I_1=\{i_k+1, \ldots, n,1, \ldots, i_1-1\}, I_j=\{i_{j-1}+1, \ldots,
i_j-1\}, j=2, \ldots, k.
\end{equation*}
These are the paths (i.e. consecutively connected vertices) on $R_n$ consisting of vertices which remain
inactive after the 1-st exploration phase. 
Hence,
\[ \cA_1(T_1) \cup \left( \cup_{j=1}^{k}I_j\right)=\{1, \ldots, n\},\]
and
\[k+ \sum_{j=1}^{k}|I_j|=n.\]
Define also
\[N_l=\#\{j \geq 1: |I_j|=l\}, \ l\geq 0.\]
Assuming the uniform distribution of the set $\cA_1(T_1)$, we derive
for all $l$ such that $l\leq n-k$
\begin{equation}\label{PIl}
\P \{|I_j| =  l  \}=\frac{\left(
  \begin{array}{c}
    n-2-l\\
    k-2
    \end{array}
\right)
}{\left(
  \begin{array}{c}
    n-1\\
    k-1
    \end{array}
\right)
}.
\end{equation}
In particular, this yields 
\begin{equation*}\label{I1}
\P \{|I_j|= 1 \}
=\frac{n-k}{n-1}\ \frac{k-1}{n-2},
\end{equation*}
and
\begin{equation}\label{I2}
\P \{|I_j| \leq  2  \}= 3\frac{k}{n} + o \left( \frac{k}{n} \right),
\end{equation}
when $k=o(n)$. We have $k=T_1 \leq \frac{2}{np^2} = o\left( \frac{1}{p}\right) = o(n)$ since $p \gg \frac{1}{n}$.

Recall that any vertex of any  $I_j$
has one mark with probability defined by \eqref{pixi}
\begin{equation}\label{p1}
p_1 :=  \frac{kp}{1+(k-1)p},
\end{equation}
independent of the other vertices.

For all $l>1$ and $j\geq 1$, given that $|I_j|=l$,
let $M_{j}(l)$ be the (random) number of
vertices in $I_j$ which have a mark and which are either
the
endpoints of $I_j$ or they are
connected in $R_n$  (i.e.,  through the deterministic edges) to the
endpoints of  $I_j$ through vertices with marks. Observe that
only in the case when $M_{j}(l)=l-1$, the remaining 
inactive vertex of the path $I_j$ has 2 active $R_n$-neighbours and
it will become active as well by the end of the expansion phase, by step 3 of the phase.
 This leads to the
following representation of the number of vertices in the set $\cD_1$:
\begin{equation}\label{D1}
|\cD_1|=N_1 + \sum_{l>1}\sum_{j\geq 1} \ind \{|I_j|=l\}(M_{j}(l)
+ \ind \{M_{j}(l)=l-1\}).
\end{equation}

Note that  the distribution of $M_{j}(l)$ does not depend on $j$; we set
$M(l) \stackrel{d}{=}M_{j}(l)$.
It is straightforward to derive for all $l\geq 2$
\begin{equation*}\label{pM}
\P \{M(l)  \geq l-1\}
= p_1^l + l(1-p_1) p_1^{l-1},
\end{equation*}
and for all $0<m\leq l-2$
\begin{equation*}\label{pM1}
\P \{M(l)  =m \}
=(m+1)p_1^{m}(1-p_1)^2.
\end{equation*}

We shall also define a random variable  $M_{j}(|I_j|)$ which,
conditionally on $|I_j|=l$, has the same distribution as $M_{j}(l)$.
In particular,
\begin{equation}\label{pM2}
\P \{M_{j}(|I_j|)  =1 \}= 2p_1(1-p_1)^2\P \{|I_j|>2 \} + \P \{|I_j|=1\}.
\end{equation}
Now we can rewrite (\ref{D1}) as 
\begin{equation}\label{D12}
|\cD_1|= \sum_{j\geq 1} \ind \left\{ \{ M_{j}(|I_j|)=1\}\cap \{ |I_j|>2\}  \right\} + {\cal R}=:D + {\cal R},
\end{equation}
where
\[{\cal R}= N_1 +\sum_{l>1}\sum_{j\geq 1} \ind \{|I_j|=l\}(M_{j}(l)
+ \ind \{M_{j}(l)=l-1\}) \ind \{M_{j}(l)>1\}
\]
\[+ 2 \sum_{j\geq 1} \ind \{|I_j|=2\}\ind \{M_{j}(l)=1\}.
\]
Compute now
\begin{equation}\label{ED1}
\E \{{\cal R}  \mid N_1, \ldots ,N_n\}  
\end{equation}
\[
=
N_1 +  2 N_2 \P \{ M(2) >0\} 
+  \sum_{l\geq 3}  N_l
\left( \sum_{m=2}^{l-2}m \P \{ M(l)  =m \}
  +l\P \{M(l)  \geq l-1\} \right)\]
\[ = N_1 +  2 N_2 (2p_1- p_1^2)+  3 N_3 \P \{M(3)  \geq 2\} \]
\[ +   \sum_{l\geq 4}  N_l
\left( (1-p_1)^2\sum_{m=2}^{l-2}m(m+1)p_1^{m}
  +l(l(1-p_1) p_1^{l-1} +  p_1^l) \right)\]
\[
 = N_1 +  2 N_2 (2p_1- p_1^2)+  3 N_3(p_1^3+3(1-p_1)p_1^2 )
\]
\[
 +   \sum_{l\geq 4}l  N_l (1-p_1)^2
\left(
l(1-p_1) p_1^{l-1} +   p_1^l  \right)  +  \sum_{l\geq 4}  N_l (1-p_1)^2\sum_{m=2}^{l-2}m(m+1)p_1^{m}
\]
\[
 \leq N_1 +  4 N_2 p_1 +  9 N_3 p_1^2 
+  \left( \max_{l\geq 4} l^2p_1^{l-1}\right) \sum_{l\geq 4}  N_l  +
(6p_1^2 +O(p_1^3 )) \sum_{l\geq 4}  N_l. \]
Since $p_1=o(1)$,  and $\sum_{l\geq 1}  N_l\leq
k$, we derive from (\ref{ED1}) with a help of (\ref{PIl}):
\begin{equation}\label{ED3}
\E \{{\cal R}   \} 
\leq  O(k^2/n)  +O(kp_1^2 ) = O(k^2/n)  +O(k^3p^2 )= O(k^2/n)=o(pk^2) . 
\end{equation}

Therefore, we have $\cR = O_{L^1} (pk^2)$ and thus $\cR = o_p (pk^2)$.
Consider now the main term in (\ref{D12}).
Let
\begin{equation*}\label{Ng2}
N_{> 2}=\sum_{i=1}^{k}\ind \{|I_i|>2\}=k- \sum_{i=1}^{k}\ind
\{|I_i|\leq 2\},
\end{equation*}
and let  $\eta_i$, $i\geq 1$, be
independent copies
of the  Bernoulli
random variable $\eta$ such that
\begin{equation}\label{eta}
\P \{\eta  =1 \} = \P \left\{M_{j}(|I_j|)  =1 \big| |I_j|>2\right\}=
2p_1(1-p_1)^2,
\end{equation}
as defined in (\ref{pM2}). Then we have the following equality in distribution:
\begin{equation*}\label{D13}
D = \sum_{j\geq 1} \ind \left\{ \{ M_{j}(|I_j|)=1\}\cap \{ |I_j|>2\}
\right\} \stackrel{d}{=}
\sum_{i=1}^{N_{> 2}}\eta_i.
\end{equation*}
%\begin{comment}
With the help of \eqref{ED3}, we deduce that 
\begin{equation}\label{Fe1}
\E \left(|\cD_1|\right) = \E (D) + \E (\cR) = 2k^2 p \etto.
\end{equation}
Thus we have $|\cD_1| = O_{L^1} (k^2p)$, moreover, if $k^2p \to \infty$ then $|\cD_1| = 2k^2p \left(1+o_{L^1}(1)\right)$.
%\end{comment}
It is straightforward to compute, taking into account (\ref{eta}) and
(\ref{I2}), that
\begin{equation}\label{ED}
\E D  =  \E \eta \E N_{> 2}= 2kp_1(1-p_1)^2(1-O(k/n)),
\end{equation}
and
\begin{align*}\label{D14}
\Var(D) & = \E \left( \Var(D\mid N_{> 2})\right) + \Var(\E (D \mid N_{> 2}))
\\ & = \Var \eta \E N_{> 2} + ( \E \eta)^2 \Var  N_{> 2}\leq 2p_1  k  +
(2p_1)^2  k ^2  O\left( \frac{ k }{n}\right).
\nonumber
\end{align*}
The last bound under assumption 
$k\leq 2/(np^2)$ and $p\geq n^{-1}$ yields
\begin{equation}\label{D151}
\Var(D)= 
O\left( p_1 k+  (p_1k)^2 \frac{k}{n} \right), \mbox{ if } p_1 k \rightarrow \infty.
\end{equation}
Now using  (\ref{D12}) we have
\[|\cD_1|= D + \cR = \E D + 
O_{L^1}\left( \sqrt{ \Var(D)}  \right)+ O_{L^1}\left( \E (\cR) \right) ,\]
which 
 together with 
(\ref{ED3}), (\ref{ED}) and   (\ref{D151})
confirms that 
 \[
|\cD_1|= 
2pk^2 + O_{L^1}\left( \sqrt{pk^2} + \sqrt{\frac{k}{n}}
pk^2 + {\frac{k^2}{n}}
\right) ,  \mbox{ if } p k^2 \rightarrow \infty. 
\]
Taking again into account that $k\leq 2/(np^2)$ 
we derive from here
\begin{equation}\label{Fe2}
|\cD_1|= 
2pk^2
+
O_{L^1}\left( \frac{k^2}{n} + \sqrt{pk^2} 
\right) ,  \mbox{ if } p k^2 \rightarrow \infty. 
\end{equation}
In the case when $p k^2$ is bounded, we have simply by (\ref{Fe1})
that 
\[|\cD_1|= O_{L^1}\left( {pk^2}\right).\]
This together with (\ref{Fe2}) finishes the proof of the lemma. 
\end{proof}

\begin{cor}\label{Cor1} Let  $\cA_1(T_1)$
be a set of vertices uniformly distributed on $V=\{1,\ldots, n\}$. Given  that
 $|\cA_1(T_1)| = k = O \left( \frac{1}{np^2} \right)$, the following holds:
\begin{romenumerate}
\item \label{Cor1i} if  $  n^{-1} \ll p \ll  n^{-2/3}$, then 
\[
 |\cD_1| 
=  2pk^2 \left(1+ o_{L^1}(1) \right);
\]

\item \label{Cor1ii} if  $  pn^{2/3} \to const >0$, then 
\[
 |\cD_1| =  O_{L^1}(1);
\]

\item \label{Cor1iii} if  $  n^{-2/3} \ll p \ll  n^{-1/2}$, then 
\[
 |\cD_1| =  o_{L^1}(1).
\]
\end{romenumerate}
\end{cor}
\hfill$\Box$

\begin{rem}\label{rem:theo2.2ii}
Notice that $|\cD_1| = o_{L_1} (1)$ in Corollary \ref{Cor1} \ref{Cor1iii} implies that 
\whp \ the process of bootstrap percolation stops after the first expansion phase. 
This allows us to  prove Theorem \ref{T2} \ref{T2ii}.
If $\frac{A_0-\acx}{\sqrt{\ac}} \to - \infty$, then by 
Theorem 3.6  (i) \cite{JLTV}  \whp{}  $|\cA_1(T_1)| = 
O \left( \frac{1}{np^2} \right)$.
Hence, if $  n^{-2/3} \ll p \ll  n^{-1/2}$  Corollary \ref{Cor1} \ref{Cor1iii} yields
that  \whp{} $\Ax = A(T_1) \leq \tc (1+\epsilon)$.
\end{rem}

For the remaining (the second and further on) expansion phases 
we will  need only the upper bounds for the number of activated vertices in the subcritical case.
\begin{lem}\label{CJ2}
Let $n^{-1}\ll p \ll n^{-2/3}$. 
Then for any $k>1$ given
 $\sum_{l=1}^kT_l< \frac{3}{np^2}$, one has
\begin{equation*}\label{F1}
  |\cD_k| \leq 
\left\{
\begin{array}{l}
4p T_k\sum_{l=1}^kT_l   +
 O_{L^1} \left( \frac{pT_k\sum_{l=1}^kT_l }{pn} +
\sqrt{pT_k\sum_{l=1}^kT_l}
\right), \
    \mbox{ if } 
p\left(T_k \sum_{l=1}^kT_l\right) \rightarrow \infty, \\ \\
O_{L^1}\left( p  T_k \sum_{l=1}^kT_l  \right),  \text{ otherwise}. 
\\
\end{array}
\right.
\end{equation*}
\end{lem}

\begin{proof}[Proof of Lemma \ref{CJ2}]
Assume we are given the sets $\cA_1(T_1), \ldots, \cA_k(T_k)$. 
Recall that after the $k$-th expansion phase, the set of remaining inactive vertices 
forms  intervals on $R_n$ with the following properties: the end 
 points of each interval do not have marks from the sets 
$\cA_1(T_1), \ldots, \cA_{k-1}(T_k-1)$ but may have at most one mark from the set $\cA_k(T_k)$ and the rest of the points of the intervals may have at most one mark from the sets $\cA_1(T_1), \ldots, \cA_{k}(T_k)$. Recall, that a vertex has a mark from a set, if it is connected by a random edge with this set.

Notice, that 
$\cA_{k}(T_k)$ is distributed uniformly on the remaining 
$n-(T_1+ \ldots  +T_{k-1})$ vertices, and  $|\cA_{k}(T_k)|=T_k$. Hence, 
there are at most $2T_k$ vertices on the boundary of $\cA_{k}(T_k)$ denoted $\partial_1 \left( \cA_k(T_k)\right)$
and each of these may have at most one mark with a probability at most
$p\sum_{l=1}^{k} T_l.$ Denote $D_k^1$ the number of the nodes on the outer boundary of $\cA_{k}(T_k)$ which have one mark. 

Furthermore, there are at most $2\sum_{l=1}^{k-1}T_l$ vertices on the boundary of $\cup_{l=1}^{k-1}\cA_{l}(T_l)$, each of which may have at most  one mark (from the set $\cA_k(T_k))$ with a probability at most
$pT_k.$ Denote $D_k^2$ the number of the nodes on the boundary of 
$\cup_{l=1}^{k-1}\cA_{l}(T_l)$
which have one mark. 

In order to get an upper bound for
 $|\cD_k|=D_k^1+D_k^2$,
 we may now almost  repeat the proof of Lemma \ref{lem1} twice to get the bounds for each $D_k^1$ and $D_k^2$ separately: first time we 
 replace $p_1$ (see (\ref{p1})) by $p\sum_{l=1}^kT_l$ 
and $T_1$ by $T_k$, and the second time we 
 replace $p_1$ by $pT_k$  and $T_1$ by $\sum_{l=1}^{k-1} T_l$. This gives us Lemma \ref{CJ2}.
\end{proof}

%%%%%%%%%%%%%%%%%%%%%%%%%%%%%%%%%%
\subsection{The number of vertices activated in an exploration phase.}

Let us fix $k\geq 1$ arbitrarily. The $k$-th  expansion phase leaves us with
 the set ${\cal A}^k$ of $T_1+ \ldots + T_k $ used active vertices and a set $\cD_k$
of unused active vertices.
We shall consider here only the values
\begin{equation}\label{Ass1}
 T_1+ \ldots +T_k \leq 3t_c=\frac{3}{np^2}.
\end{equation}
(Observe that if (\ref{Ass1}) does not hold then almost percolation happens even  on
 the edges of $G_{n,p}$ only, see \cite{JLTV}).
Also, we shall assume that $  n^{-1} \ll p \ll  n^{-2/3}$, which by the Corollary \ref{Cor1} implies that $|\cD_1|$ is large \whp 

Consider now the $k+1$-st exploration
  phase.
By the definition (\ref{Sbin*})
we have
\begin{align}\label{Sbin1}
 |{\cal A}_{k+1}(t)| & =|\cD_k|+
 S_{k+1}(t)
\\
&  \stackrel{d}{=}|\cD_k|+
 \Bin  \left(n-\sum_{l=1}^kT_l -|{\partial}_1(\cD_k
   \cup {\cal A}^k)|,\pi_{k+1}(t)\right) + \Bin  \left(|{\partial}_1(\cD_k
   \cup {\cal A}^k)|,\pi_{1}(t)\right),
\nonumber
\end{align}
where
  \begin{equation}\label{pi2}
\pi_{k+1}(t)=\pi_{1}(t) +
  \frac{  p\sum_{l=1}^kT_l}{1+(\sum_{l=1}^kT_l-1)p}(1-p)^{t-1}pt=:\pi_{1}(t) +
  \pi_+(t).
\end{equation}
Notice, that under assumption (\ref{Ass1}) we have the following bounds
for all $t=o(1/p)$:
\begin{equation}\label{Fe3}
\pi_{1}(t)= O(p^2t^2)
\end{equation}
and
\begin{equation}\label{Fe4}
\pi_{k+1}(t)=O\left( p^2t^2 + p^2t/(np^2)\right) =O\left( p^2t^2 + 
t/n \right) .
\end{equation}

Given $T_1, \ldots T_k$, $\cD_k$ and set ${\partial}_1(\cD_k
   \cup {\cal A}^k)$
we shall approximate the terms  in (\ref{Sbin1}) separately. 
   Let us define two processes
 \begin{equation*}\label{SP}
   S^{(1)}(t):=\Bin  \left(K_1,\pi_{1}(t)\right), \ \ S^{(2)}(t):=\Bin  \left(K_{2},\pi_{k+1}(t)\right), \ \ 
\end{equation*}
 where
 \begin{equation}\label{K}
K_1:= |{\partial}_1(\cD_k
   \cup {\cal A}^k)|, \  \  \ K_{2}= n-\sum_{l=1}^kT_l -|{\partial}_1(\cD_k
   \cup {\cal A}^k)|.
\end{equation}

 \begin{prop}\label{P1} Given numbers $K_1$ and $K_2$ the processes
\begin{equation*}\label{eq:martin}
\frac{S^{(1)}(t)-\E S^{(1)}(t)}{1-\pi_{1}(t)}, \ \ 
\frac{S^{(2)}(t)-\E S^{(2)}(t)}{1-\pi_{k+1}(t)},
\end{equation*}
$t=0,1, \ldots, $   are martingales.
   \end{prop}
   
   \begin{proof}[Proof of Proposition \ref{P1}]
   For the process $S^{(1)}(t)$ the proof is the same as
     for Lemma 7.2 in \cite{JLTV}. It is practically the same for the
     process $S^{(2)}(t)$ as well, which we explain now.
Note that $S^{(2)}(t)$ is a sum of $i.i.d.$ processes so that
\[
S^{(2)}(t)=\sum_{v=1}^{K_2}\ind \{\xi_v+ \sum_{j=1}^t\xi_{jv}\geq 2\},
\]
where $\xi_v, \xi_{jv},$ $v\geq 1$, $j\geq 1$ are independent
Bernoulli random variables, such that $\xi_v \in Be(p_+)$ with
\[p_+:= \frac{\sum_{l=1}^kT_lp}{1+(\sum_{l=1}^kT_l-1)p},\]
and $\xi_{jv} \in Be(p). $ 
Then it is straightforward to check that
\[X_v(t):=   \frac{ \ind \{ \xi_v+ \sum_{j=1}^t\xi_{jv}\geq 2\} -
  \pi_{k+1}(t)  }{1-\pi_{k+1}(t)}
\]
is a martingale, taking also into account that
\[\pi_{k+1}(t)=\P\{\xi_v+ \sum_{j=1}^t\xi_{jv}\geq 2 \}= \P\{X_v(t)=1\}.\]
Then 
\[\frac{S^{(2)}(t)-\E S^{(2)}(t)}{1-\pi_{k+1}(t)}=
\sum_{v=1}^{K_2}X_v(t)  
\]
is also a martingale. 
\end{proof}

Since $K_i\leq n$ for $i=1,2$, we can make use of the properties of  martingales to get immediately the following bounds. 

\begin{cor}(Lemma 7.3 \cite{JLTV})\label{CF} For any $t_0$,
  \[
 \E\left(\sup_{t\leq t_0} |S^{(1)}(t)-\E S^{(1)}(t)| \right)^2 \leq
 4\frac{n \pi_{1}(t_0)}{1-\pi_{1}(t_0)}, \]
\[    \E\left(\sup_{t\leq t_0} |S^{(2)}(t)-\E S^{(2)}(t)| \right)^2 \leq
 4\frac{n \pi_{k+1}(t_0)}{1-\pi_{k+1}(t_0)}.\]
\hfill$\Box$
  \end{cor}

For all $t\leq t_0\leq o(1/p)$, when in particular, $\pi_i(t)=o(1)$,
the bounds from Corollary \ref{CF} yield the following approximation
\begin{align*}\label{ME4}
  S_{k+1}(t)& =S^{(1)}(t)+S^{(2)}(t) 
  \nonumber
\\
& = \E S^{(1)}(t) +\E S^{(2)}(t)
  +O_{L^2}\left(\sqrt{n  \pi_1(t_0) } +  \sqrt{n\pi_{k+1}(t_0) } \right).
  \end{align*}
Combining this with (\ref{Sbin1}) we obtain for all $t\leq t_0\leq o(1/p)$
\begin{equation}\label{F3}
 |{\cal A}_{k+1}(t)| =|\cD_k|+
\E S^{(1)}(t) +\E S^{(2)}(t)
  +O_{L^2}\left(\sqrt{n\pi_1(t_0)} +\sqrt{n\pi_{k+1}(t_0)} \right).
\end{equation}

We begin with the asymptotics of the number of activated vertices in 
 the second exploration phase. 
 As we will see, under the conditions of Theorem \ref{T2} \ref{T2i}, 
(almost) percolation happens during the second exploration phase. Therefore, we concentrate on this phase and prove Theorem \ref{T2} \ref{T2i}.

\begin{lem}\label{LJ1}
Let  $n^{-1}\ll p\ll n^{-1/2}$. Then given $T_1=O(t_c)$ for all  $t\leq t_0 = O(t_c)$
one has
\[
|{\cal A}_2(t)| =  \left(n-T_1\right)\frac{(tp)^2}{2} +
  \left(n-3T_1\right)p^2tT_1
 + |\cD _1| -n\frac{(tp)^3}{3}(1+o(1))
\]
\[+ O(t/t_c)+o(p(t^2+T_1^2))+O_{L^2}\left( \sqrt{t_0} \right).\]
\end{lem}
\noindent
{\bf Proof.}
First we derive from (\ref{F3})
\begin{equation}\label{F3M}
 |{\cal A}_{2}(t)| =
 |\cD _1|+
\left(n-T_1-\E  \left\{ |{\partial}_1(\cD _1
   \cup {\cal A}^1)|
\mid  T_1 \right\} \right) \pi_{2}(t) 
\end{equation}
\[+\E  \{ |
{\partial}_1(\cD _1
   \cup {\cal A}^1)|| T_1\}  \pi_1(t)+O_{L^2}\left(\sqrt{n\pi_1(t_0)} +\sqrt{n\pi_{2}(t_0)} \right).\]
   Consider ${\partial}_1(\cD _1
   \cup {\cal A}^1)$. Since the vertices of $\cD _1$ are connected to the
   boundary of ${\cal A}^1$, we have $|{\partial}_1(\cD _1
   \cup {\cal A}^1)|\leq |{\partial}_1(
   {\cal A}^1)|$. When an entire  interval of  inactive vertices  becomes
   active after the $1$-st expansion phase, the boundary of the active
   set looses exactly 2 vertices if this interval has at least 2 vertices, otherwise, 
it loses 1 vertex. Hence, using again sets $I_i$, $i=1,
   \ldots, T_1$, defined
   in the proof of Lemma \ref{lem1} , we get the following representation
\begin{equation}\label{bA1}
|{\partial}_1(\cD _1
   \cup {\cal A}^1)| = |{\partial}_1(
   {\cal A}^1)| - N_1- 2\sum_{j:|I_j|\geq 2}{\bf I}\{M_j(|I_j|)\geq |I_j|-1\}.
\end{equation}
Since
\[
 |{\partial}_1(
   {\cal A}^1)| = N_1 +  2\sum_{j\geq 1}{\bf I}\{|I_j|\geq 2\}=  N_1 +
   2\left( T_1-  N_1 - N_0 \right),
\]
we derive 
\begin{equation}\label{bA2}
|{\partial}_1(\cD _1
   \cup {\cal A}^1)| = 2T_1
   -2N_1- 2 N_0 - 2\sum_{j:|I_j|\geq 2}{\bf I}\{M_j(|I_j|)\geq l-1\}.
\end{equation}
With the same argument as we derived Lemma \ref{lem1} we get from here
\begin{equation}\label{bA3}
|{\partial}_1(\cD _1
   \cup {\cal A}^1)| = 2T_1(1+o_{L^1}(1)) .
\end{equation}

Using also  approximations
\begin{equation}\label{pi1}
  \pi_1(t)= \P\{\Bin(t,p)\geq 2\}=
\frac{(tp)^2}{2}-\left(t \frac{p^2}{2}+ \frac{(tp)^3}{3} \right)(1+o(1)),
\end{equation}
 and (see (\ref{pi2}) with $k=1$)
 \begin{equation}\label{pi2+} 
 \pi_+(t) = p^2tT_1(1+o(p)(T_1+t))=p^2tT_1 + o(p^3tT_1)(t+T_1),
\end{equation}
we derive from (\ref{F3M})  that 
\[
|{\cal A}_2(t)| = 
|\cD _1|+
\left(n-T_1\right) \pi_{1}(t) +
\left(n-T_1-\E  \left\{ |{\partial}_1(\cD _1
   \cup {\cal A}^1)|
\mid  T_1 \right\} \right) \pi_{+}(t) 
\]
\[
+O_{L^2}\left(\sqrt{n\pi_1(t_0)} +\sqrt{n\pi_{2}(t_0)} \right)\]

\[
=  |\cD _1|+\left(n-T_1\right)
 \left( \frac{(tp)^2}{2}-\left(t \frac{p^2}{2}+ \frac{(tp)^3}{3} \right)(1+o(1))
\right)\]
 \[
+
  \left(n-3T_1\right)\left( 
p^2tT_1 + o(p^3tT_1)(t+T_1)
\right)
+
 o(p T_1(t+T_1))
 +O_{L^2}\left( \sqrt{np^2t_0^2} \right)
\]

\begin{equation}\label{ME6}
=|\cD _1|+ \left(n-T_1\right)\frac{(tp)^2}{2} +
  \left(n-3T_1\right)p^2tT_1 -n\frac{(tp)^3}{3}
\end{equation}
 \[+no(tp)^3+nO(tp^2)+ no(p^3tT_1)(t+T_1)
+
 o(p T_1)(t+T_1)
 +O_{L^2}\left( \sqrt{t_0} \right).
\]
This  yields the statement of the Lemma. \hfill$\Box$

Using the result of  Lemma \ref{LJ1} combined with bound from Lemma 
\ref{LF1} 
consider now function
\[
A_2(t)-t=\left(
   (n-T_1)\frac{p^2}{2}t^2
+ 
 \left( np^2 T_1-1-  3  (T_1 p)^2 
 \right)t + 2pT_1^2\right)-n\frac{(tp)^3}{3}(1+o(1))+{\cal R}_{T_1}(t)\]
\begin{equation}\label{di1}
=: f_{T_1}(t)-n\frac{(tp)^3}{3}(1+o(1))+ {\cal R}_{T_1}(t).
 \end{equation}
where for all  $t\leq t_0 = O(t_c)$
\begin{equation}\label{J13}
{\cal R}_{T_1}(t)= O(t/t_c)+o(p(t^2+T_1^2))+O_{L^1}\left( \sqrt{t_c} + \frac{T_1^2}{n}\right).
\end{equation}
Notice here that when $t\leq T_1=O(1/(np^2))$  we have by (\ref{J13})
\begin{equation}\label{bR}
{\cal R}(t)=
O_{L^1}\left(\frac{1}{\sqrt{np^2}}\right) + o\left( \frac{1}{n^2p^3}\right).
 \end{equation}

We shall study  the minimal value of the introduced above function
\begin{equation}\label{minf}
f_{T_1}(t)=  (n-T_1)\frac{p^2}{2}t^2
+ 
 \left( np^2 T_1-1-  3  (T_1 p)^2 
 \right)t + 2pT_1^2 .
 \end{equation}
 First we observe that the argument of the minimal value of this function is
 \begin{equation}\label{tmin}
t_{min}:=\frac{1- np^2 T_1+  3  (T_1 p)^2 }{(n-T_1) p^2},
 \end{equation}
 and
\begin{equation}\label{cc3}
 \ t_{min}<0 \ \  \Leftrightarrow \ \ np^2 T_1>1+  3  (T_1 p)^2 .
\end{equation}
Then 
\begin{equation}\label{min}
\min_{0\leq t\leq n}f_{T_1}(t) = 
\left\{
\begin{array}{ll}
2pT_1^2 - \frac{\left(np^2 T_1-1- 3  (T_1 p)^2 
 \right)^2 }{(n-T_1) p^2}, & \mbox{ if } np^2 T_1<1+  3  (T_1 p)^2 ,\\ \\ 
2pT_1^2 , & \mbox{ otherwise. }
\end{array}
\right.
 \end{equation}

\subsection{Critical case: proof of Theorem \ref{T2}.}

Let us recall one more result from \cite{JLTV} which describes the critical case of bootstrap percolation on $G_{n,p}$.
\begin{theo*}[Theorem 3.8  \cite{JLTV}]
Suppose that $n\qw\ll p\ll n^{-1/2}$. Let $A^*_0$ be the total number of
vertices activated due to a bootstrap percolation (with threshold
$r=2$) on a random graph $G_{n,p}$ starting with $A_0$ active vertices. 
  
If  $A_0/a_c \to 1$ and also $(A_0-\acx)/\sqrt{\ac}\to -\infty$, then $A^*_0$ is asymptotically normal with the following parameters
\[A^*_0 \in \mbox{AsN}\left(t_*, \frac{t_c}{2(1-A_0/a_c^*)} \right),\]
where $ t_* =t_c + p\tc^2 \etto - \sqrt{2t_c(a_c^*-A_0)}\etto$.
\end{theo*}

Assume now in the conditions of Theorem \ref{T2}
that    for some $\omega(n) \rightarrow \infty $ but such that $\omega(n)=o(\ac)$
we have
\begin{equation}\label{Ar1}
\frac{\acx-A_0}{\omega(n)\sqrt{\ac}} \to 1.
\end{equation}
This implies by the cited above 
Theorem 3.8  \cite{JLTV}
that 
\[
  T_1= \frac{1}{np^2} + O\left( p\frac{1}{(np^2)^2} \right)+ 
O\left( \sqrt{\frac{\acx- A_0}{np^2}} \right) + 
O_P\left( \sqrt{ \frac{1}{ np^2 (1-A_0/\acx)}  }\right) 
\]
\begin{equation}\label{J2}
=  \frac{1}{np^2} + 
O_P\left(  \frac{1}{n^2p^3} + \left( \frac{1 }{np^2}\right)^{3/4}  \sqrt{ \omega(n) }
 \right).
\end{equation}
Substituting this into  (\ref{min})  we derive             
\begin{equation}\label{min1}
\min_{0\leq t\leq n}f_{T_1}(t) = 2pT_1^2 + O_P\left( \frac{1}{n^3p^4} +
\frac{\omega(n) }{\sqrt{ n}p} 
 \right).
\end{equation}
Substituting now (\ref{min1}) and (\ref{bR}) into (\ref{di1}) ,  we get
 for all $t \leq c \frac{1}{np^2} $, where $1<c<6^{1/3}$
 $$A_2(t)-t \geq  2\frac{1}{n^2p^3}  (1+o_P(1))
+ O_P\left( 
\frac{\omega(n) }{\sqrt{ n}p} 
 \right) -n\frac{(tp)^3}{3} (1+o_P(1)) + o_P\left(\frac{1}{n^2p^3} \right)
$$
\begin{equation}\label{Al}
\geq 
\frac{(6-c^3)}{3n^2p^3}(1+o_P(1)) + O_P\left( 
\frac{\omega(n) }{\sqrt{ n}p} 
 \right).
\end{equation}
Hence, only if $p=o\left(n^{-3/4}\right)$  we can choose $\omega(n)=o\left( 
{pn^{3/4}} 
 \right)^{-2} $ so that $\omega(n) \rightarrow \infty$. This choice
 by 
(\ref{Al})
will give us
$A_2(t)-t \gg 0 $ for all $t \leq c \frac{1}{np^2} $.

We conclude that in the 2-nd exploration phase the process accumulates \whp
\ at least $ct_c=c \frac{1}{np^2} $, $c>1$, active vertices and  passes value $t_c$, which is
 the critical value  for the $G_{n,p}$. From this state on the process 
$A_2(t)$ evolves to the state $n-o(n)$ by Lemma 8.2 from \cite{JLTV} (More precisely, Lemma 8.2 \cite{JLTV} is proved under assumption that 
the process accumulates \whp  \ $3t_c$ active vertices. However, the proof is easy to modify in order to replace  $3t_c$  by  $ct_c$ for any $c>1$). This proves statement $(i)$ of Theorem \ref{T2}.

Assume now that $pn^{2/3}\rightarrow \infty$. 
The statement \ref{T2ii} of Theorem \ref{T2} follows by the assertion 
\ref{Cor1ii} of Corollary \ref{Cor1} as we explained in Remark \ref{rem:theo2.2ii}.
\hfill$\Box$

\subsection{Proof of Theorem \ref{T1}.}

\subsubsection{Subcritical case}
\begin{lem}\label{LF1}
Let   $n^{-1}\ll p\ll n^{-1/2}$, 
and let $k>1$ be fixed arbitrarily.
Under assumption that $T_1, \ldots , T_k$ are given such that
 $\sum_{l=1}^kT_l <\beta t_c$ for some $\beta<1$ we have the following.

(i)  If $pT_k \sum_{l=1}^k T_l =O(1)$, then $T_{k+1} = O_{L^1}(1)$;

(ii) otherwise, if $pT_k \sum_{l=1}^k T_l \to \infty$, then 
\begin{equation}\label{F5}
T_{k+1}\leq 
 \frac{10}{1-\beta} p T_k \sum_{l=1}^k T_l 
\end{equation}
with probability at least 
\begin{equation}\label{F23}
1- O\left(
 \frac{1}{np}
+ \left( p  T_k \sum_{l=1}^kT_l \right) ^{-1} 
\right).
\end{equation}
\end{lem}

\begin{proof}[Proof of Lemma \ref{LF1}]
First we derive from (\ref{F3}), taking into account (\ref{pi2})  that for all $t\leq t_0=o(1/p)$
\begin{equation}\label{F7}
 A_{k+1}(t)=|{\cal A}_{k+1}(t)| \leq |\cD _k|+
n(\pi_{1}(t) +  \pi_+(t)) +
O_{L^2}\left(\sqrt{n\pi_1(t_0)} +\sqrt{n\pi_{k+1}(t_0)} \right). 
\end{equation}
This together with (\ref{pi1}) and (\ref{pi2})  gives us  for all $t\leq t_0=o(1/p)$
\begin{equation}\label{F6}
  A_{k+1}(t)\leq |\cD _k|+
n\frac{(tp)^2}{2} +  tnp^2\sum_{l=1}^kT_l +
O_{L^2}\left(p\sqrt{ n  t_0^2  +nt_0\sum_{l=1}^kT_l } \right). 
\end{equation}

Assume first  that   $pT_k \sum_{l=1}^kT_l =O(1)$, which  by  Lemma \ref{F1}
yields $|\cD _k|=O_{L^1}(1) $. Then by (\ref{F6}) 
and under the assumption $\sum_{l=1}^kT_l<\beta t_c$
we have for all $t\leq t_0 \leq O(t_c) =o(1/p)$
\begin{equation}\label{F9}
  A_{k+1}(t)-t \leq 
n\frac{(tp)^2}{2} -  t\left(1- \beta \right) +
O_{L^1}\left( 1+  \sqrt{t_0} \right) . 
\end{equation}
Now it is straightforward to compute (solving the quadratic equation)
 that (at least)  for all 
\begin{equation}\label{F14}
0\leq t \leq  \frac{1-\beta}{2np^2}
\end{equation}
we have in (\ref{F9})
\begin{equation}\label{F10}
  A_{k+1}(t  )-t \leq -\frac{1-\beta}{2}t    +
O_{L^1}\left(1+\sqrt{ t_0} \right) . 
\end{equation}
Hence, choosing here $t_0=O(1)$ (we do not use  (\ref{F7})
for $t\gg 1$), will  imply that $A_{k+1}(t )< t $ for some
\begin{equation}\label{F11}
  t=T_{k+1}: = O_{L^1}\left(1 \right) ,
\end{equation}
which, notice, also satisfies (\ref{F14}). This yields statement $(i)$
of the Lemma.

When  $pT_k \sum_{l=1}^kT_l \to \infty$ and $\sum_{l=1}^kT_l<\beta t_c$ 
we shall use  Lemma  \ref{F1}   to  derive from (\ref{F6})
for all $t\leq t_0=O(t_c)=o(1/p)$
  \[
A_{k+1}(t)-t 
\leq   n\frac{(tp)^2}{2} -(1 - \beta) t
+ 4p T_k  \sum_{l=1}^k T_l  
 +   O_{L^1}
\left( \frac{1}{n}T_k\sum_{l=1}^kT_l \right) + O_{L^2}
\left( 
p\sqrt{ n  t_0^2  +nt_0\sum_{l=1}^kT_l } \right) 
\]
\begin{equation} \label{F13}
 = n\frac{(tp)^2}{2} -(1 - \beta) t
+ 4p T_k  \sum_{l=1}^k T_l  + 
 O_{L^1}\left(  \frac{pT_k\sum_{l=1}^kT_l }{np}\right) + O_{L^2}
\left( \sqrt{ t_0}
\right).
\end{equation}
Then 
we derive solving the quadratic equation, that (at least)  for all
\begin{equation}\label{F18}
\frac{9}{1-\beta} p T_k \sum_{l=1}^kT_l \leq  t \leq   t_0 \leq 
 \frac{1-\beta}{np^2} 
\end{equation}
we have
\begin{equation}\label{F8}
 n\frac{(tp)^2}{2} -(1 -  \beta) t
+ 4p T_k  \sum_{l=1}^k T_l  <  - \frac{1}{2}p T_k  \sum_{l=1}^k T_l .
 \end{equation}
Therefore by (\ref{F13}) with $t_0=\frac{10}{1-\beta} p T_k \sum_{l=1}^kT_l$
for all $t$
which satisfy  (\ref{F18}), i.e., 
\[\frac{9}{1-\beta} p T_k \sum_{l=1}^kT_l \leq  t \leq  \frac{10}{1-\beta} p T_k \sum_{l=1}^kT_l,\]
it holds that
\[
A_{k+1}(t)-t \leq  -  \frac{1}{2}pT_k  \sum_{l=1}^k T_l + 
 O_{L^1}\left(  \frac{pT_k\sum_{l=1}^kT_l }{np}\right)
+ O_{L^2}\left(
\sqrt{ p T_k \sum_{l=1}^kT_l}
\right).
\]
Hence, if the right-hand side of the last formula is negative,
 the $k+1$-st exploration phase will stop at $T_{k+1} \leq 
\frac{10}{1-\beta} p T_k \sum_{l=1}^kT_l$, and thus
\[
\P 
\left\{ 
T_{k+1} \leq \frac{10}{1-\beta} p T_k \sum_{l=1}^kT_l
\mid T_k \sum_{l=1}^kT_l
\right\}
\]
\[ \geq 
\P \left\{
\min_{t\leq \frac{10}{1-\beta} p T_k \sum_{l=1}^kT_l} A_{k+1}(t)-t \leq 0 
\mid T_k \sum_{l=1}^kT_l
\right\}\]
\[ \geq 
\P \left\{ -  \frac{1}{2}pT_k  \sum_{l=1}^k T_l + 
 O_{L^1}\left(  \frac{pT_k\sum_{l=1}^kT_l }{np}\right)
+ O_{L^2}\left(
\sqrt{ p T_k \sum_{l=1}^kT_l}
\right)\leq 0 
\mid T_k \sum_{l=1}^kT_l
\right\}.\]
Using 
the Chebyshev's inequality together with the assumption that $\sum_{l=1}^k T_l \leq \beta t_c$ we derive from here
\[
\P 
\left\{ 
T_{k+1} \leq \frac{10}{1-\beta} p T_k \sum_{l=1}^kT_l
\mid T_k \sum_{l=1}^kT_l
\right\}
\]
\[ \geq 
1- \P \left\{ O_{L^1}\left(  \frac{pT_k\sum_{l=1}^kT_l }{np}\right)
> \frac{1}{4} p T_k \sum_{l=1}^kT_l
\mid T_k \sum_{l=1}^kT_l
\right\}
\]
\[
-\P \left\{ O_{L^2}\left( \sqrt{ p T_k \sum_{l=1}^kT_l} 
\right)
> \frac{1}{4} p T_k \sum_{l=1}^kT_l
\mid T_k \sum_{l=1}^kT_l
\right\} 
\]
\[ = 1-O\left(\frac{1}{np}+
\left( p  T_k \sum_{l=1}^kT_l \right) ^{-1} 
\right).
\]
This yields statement $(ii)$  of the Lemma \ref{LF1} and finishes the proof. 
\end{proof}

Consider now the relation (\ref{F5}). First we study
 a similar deterministic system.

\begin{lem}\label{LF2}
For given   $c >0$ and $t_1>0$ such that $ct_1<1$ define for $k\geq 1$
\begin{equation}\label{F21}
t_{k+1}= c
 t_k  \sum_{l=1}^k t_l .
\end{equation}
Then for any  $0<\alpha <1$ which satisfies
\begin{equation}\label{F22}
(1-\alpha)\alpha> ct_1,
\end{equation}
one has $t_k \leq \alpha^{k-1} t_1$, $k\geq 1,$ and, hence, 
\[ \sum_{l=1}^{\infty} t_l  \leq t_1\frac{1}{1-\alpha}.\]
\end{lem}
\begin{proof} [Proof of Lemma \ref{LF2}]
Write here $S_k:=\sum_{l=1}^k t_l$.
We shall show first that under condition (\ref{F22})
one has $c S_k < \alpha$ for all $k\geq 1$. 

Assume, on the contrary, that
\begin{equation}\label{F23*}
L=L(\alpha):=\max\{k: c S_k \leq  \alpha\}<\infty.
\end{equation}
By the definition (\ref{F21})
\[S_{L+1}=\sum_{l=1}^{l=L+1} t_l = t_1+ \sum_{k=1}^{L} c
 t_k  S_k, \]
where by the definition of $L$  for all $k\leq L$
\begin{equation}\label{M8}
t_{k+1}=c
 t_k  S_k \leq \alpha t_k \leq \alpha^k t_1.
\end{equation}
Hence, 
\[S_{L+1} \leq t_1+ \sum_{l=1}^{L} 
 \alpha^k  t_1 \leq  t_1\frac{1}{1-\alpha}, \]
which under condition (\ref{F22}) yields 
$cS_{L+1} < \alpha$ and thus  contradicts (\ref{F23*}). Therefore 
$c S_k \leq  \alpha$ for all $k\geq 1$. This yields (\ref{M8}) 
for all $k\geq 1$, and the statement of  the Lemma follows.
\end{proof}

We shall prove now statement \ref{T1i}  of Theorem \ref{T1}. 

Assume, that $T_1= \beta t_c$ for some $\beta <1$. Hence, $pT_1^2=\beta^2 \frac{1}{n^2p^3}$.
Consider then two cases. If $\frac{1}{n^2p^3} =O(1)$, then  by Lemma \ref{LF1}
we have $T_2=O_{L^1}(1)$, which by Lemma \ref{CJ2} implies that $E|\cD _2|
=o(1)$. Hence, \whp{} the bootstrap percolation stops after the second expansion phase.  

Assume now  that 
$\frac{1}{n^2p^3} \rightarrow \infty$, i.e., 
\begin{equation}\label{M1}
\frac{1}{n}\ll p \ll \frac{1}{n^{2/3}}.
\end{equation}
Let $h=h(n)=o(np)$ 
be an arbitrarily fixed function such that $h(n)\rightarrow \infty$. 
Notice that under assumption (\ref{M1}) condition $h(n)=o(np)$ yields
\begin{equation}\label{M14}
h(n)=o(t_c).
\end{equation}
Define
a random  time
\[\tau=\min\{k\geq 1: p  T_k \sum_{l=1}^kT_l < h\}.\]

If $pT_1^2=\beta^2 \frac{1}{n^2p^3} \rightarrow \infty $, but $pT_1^2<h$,
then by Lemma \ref{LF1} (ii) we have $T_2=O(pT_1^2)=O(h)$ \whp. Hence, \whp
\[pT_2(T_1+T_2)=O(p/(np^2))=o(1).\]
This by Lemma \ref{CJ2} implies  that  $|\cD _2|
=o(1)$ \whp, which \whp{} yields a termination of the bootstrap percolation after the second expansion phase.  

Assume, that $\beta^2 \frac{1}{n^2p^3}\geq h$, and therefore $\tau >1$. We shall get first an upper bound in probability for $\tau$.

\begin{prop}\label{Pfi}
Assume that $\beta^2 \frac{1}{n^2p^3}\geq h$ and $(np)^{np}\gg n$.
One can choose an unbounded function $h$ so that $h=o(np)$, and
for some $K_0=o(h)$
\[{\P} 
\left\{ \tau \leq  K_0+1 \right\}\geq 1-K_0 O(h^{-1}).\]
\end{prop}

\noindent
{\bf Proof.}
Recall that by Lemma \ref{LF1} for all $ k \geq 1 $  given $T_k \sum_{l=1}^kT_l>h$ we have
\begin{equation}\label{M2}
 T_{k+1} \leq \frac{10}{1-\beta} p T_k \sum_{l=1}^{k}T_l
\end{equation}
with probability at least $1-O( (np)^{-1}+h^{-1}  )$. Hence, for any $K_0$  we have
\begin{align}
 {\P} 
\left\{T_{k+1} \leq \frac{10}{1-\beta} p T_k \sum_{l=1}^{k}T_l, 1\leq k \leq K_0
 \mid  \tau > K_0 \right\}
& \geq 1-K_0 O( (np)^{-1}+h^{-1}  )
\nonumber
\\
&
=1-K_0 O( h^{-1}  ),
\label{M3}
\end{align}
where the last equality is due to the assumption that $h=o(np)$.

Let us now choose $K_0$ as follows. 
Assume that the relation  (\ref{M2})
holds for all $1 \leq
k\leq K_0$. By our assumptions we also have here
\[\frac{10}{1-\beta} p T_1=O\left(\frac{1}{np}\right)=o(1).\]
Hence, by  Lemma \ref{LF2}
with $c=\frac{10}{1-\beta} p $ we have (conditionally on (\ref{M2})  for all $1 \leq
l\leq k$)
\begin{equation}\label{M9}
T_{k} \leq \alpha^{k-1}T_1
\end{equation}
for  some $\alpha$ which satisfies condition 
$(1-\alpha)\alpha> cT_1$
(see (\ref{F22})). 
Notice, that here we can choose 
\[
\alpha < 2cT_1=2\frac{10}{1-\beta}\frac{1}{pn} =: \frac{c_1}{pn} ,
\]
which together with  (\ref{M9}) yields
\begin{equation}\label{M10}
T_{k} \leq \left( \frac{c_1}{pn} \right)^{k-1}T_1 = 
\beta \left( \frac{c_1}{pn} \right)^{k-1}  \frac{1}{p^2n},
\end{equation}
as well as 
\begin{equation}\label{M17}
\sum_{l\leq k}T_l<\frac{T_1}{1-(c_1/pn)}.
\end{equation}
This implies 
\begin{equation}\label{M18}
p T_{k} \sum_{l=1}^{k}T_l
 \leq p \beta 
\left( \frac{c_1}{pn} \right)^{k-1}  
\frac{1}{p^2n}  \frac{T_1}{1-(c_1/pn)}<\beta^2 
\left( \frac{c_1}{pn} \right)^{k}  
\frac{1}{p^2n}.
\end{equation}
Setting  now
\begin{equation}\label{M11}
K_0:= \min \left\{k: \beta^2 \left( \frac{c_1}{pn} \right)^{k+1}  \frac{1}{p^2n}
<h\right\},
\end{equation}
we have by (\ref{M18})
\begin{equation}\label{M11*}
p T_{K_0+1} \sum_{l=1}^{K_0+1}T_l
 <h.
\end{equation}

{\it Claim.} One can choose an unbounded function $h$ so that $h=o(np)
$ and 
\begin{equation}\label{Ap1}
K_0=o(h)
\end{equation}

if and only if $n=o\left( (np)^{np}\right)$.

\begin{proof}[Proof of the Claim.] Assume that some function $h\rightarrow \infty$
satisfies (\ref{Ap1}), which by  definition 
(\ref{M11}) is equivalent to 
\[(pn)^{o(h)}=\frac{n}{h}. \]
Under the assumption $h=o(np)$ and $pn, h\rightarrow \infty$, this holds if and only if  
\begin{equation}\label{Ap20}
 o(h)\log(pn)=\log n - \log h.
\end{equation}
Again under the condition that $h=o(np)$, relation (\ref{Ap20}) is equivalent to 
\begin{equation}\label{Ap21}
 o(h)=\frac{\log n}{\log(pn)} .
\end{equation}
Finally, the last equality is satisfied for some $h=o(np)$ if and only if 
\[\frac{\log n}{\log(pn)} =o(np).\]
The assertion of the claim follows.
\end{proof}

Observe that for $K_0$ which is chosen according to (\ref{M11}) and, hence, (\ref{M11*}), we have
\begin{align*}  \P &
\left\{T_{k+1} \leq \frac{10}{1-\beta} p T_k \sum_{l=1}^{k}T_l, 1\leq k \leq K_0
 \mid  \tau > K_0 \right\}
\\
& \qquad  \qquad =\P 
\left\{\left(
T_{k+1} \leq \frac{10}{1-\beta} p T_k \sum_{l=1}^{k}T_l, 1\leq k \leq K_0\right) \cap  \left(  \tau=K_0+1\right)
 \mid  \tau > K_0 \right\}
\\
&  \qquad  \qquad \leq \P 
\left\{  \tau=K_0+1
 \mid  \tau > K_0 \right\}=\P 
\left\{  \tau=K_0+1\right\}/ \, \P 
\left\{ 
   \tau > K_0 \right\}.
\end{align*}
Combining this with (\ref{M3}) we get
\[  
{\P} 
\left\{  \tau>K_0\right\}-{\P} 
\left\{  \tau>K_0+1\right\}
=
 {\P} 
\left\{  \tau=K_0+1\right\}
\geq \left(1-K_0 O( h^{-1}  ) \right){\P} 
\left\{ 
   \tau > K_0 \right\},
\]
which yields
\begin{align}\label{M19}
\P 
 \left\{  \tau \leq K_0+1\right\}=1-\P 
\left\{  \tau>K_0+1\right\}
& \geq 1-K_0 O( h^{-1}  ) \P 
\left\{ 
   \tau > K_0 \right\}
\\
& \geq 1-K_0 O( h^{-1}  ) .
\nonumber
\end{align}
This together with the assertion of the claim completes the proof of the Proposition.
\hfill$\Box$

We shall finish now the proof of the assertion \ref{T1i} of Theorem \ref{T1}.
Using the representation
$|{\cal A}^*|=\sum_{k=1}^K T_k$
consider  for an arbitrarily fixed $0<\varepsilon<1-\beta$
\begin{align}\label{M12}
\P \Bigl\{|{\cal A}^*|<
& (\beta + \varepsilon)t_c \mid T_1= \beta t_c\Bigr\}
\\
& ={\P}\left\{\sum_{l\leq K}T_l<
(\beta + \varepsilon)t_c \mid T_1= \beta t_c\right\} 
\nonumber
\\
 & \geq 
{\P}\left\{\left( \sum_{l\leq K}T_l<
(\beta + \varepsilon)t_c\right) \cap \left(\sum_{l\leq \tau}T_l<
\left(\beta + \frac{\varepsilon}{2}\right)t_c \right) \mid T_1= \beta t_c\right\}
\nonumber
\\ & 
\geq 
{\P}\left\{\left( \sum_{l\leq K}T_l<
(\beta + \varepsilon)t_c\right) \cap \left( K=\tau+1\right)
\mid \sum_{l\leq \tau}T_l<
\left(\beta + \frac{\varepsilon}{2}\right)t_c,  T_1= \beta t_c\right\}
\nonumber
\\ & \qquad \qquad  {\P}\left\{\sum_{l\leq \tau}T_l<
\left(\beta + \frac{\varepsilon}{2}\right)t_c \mid  T_1= \beta t_c\right\}.
\nonumber
\end{align}

By Lemma 
\ref{LF1} if $pT_{\tau}\sum_{l\leq \tau}T_l=O(1)$ we have 
$T_{\tau+1}=O_{L^1}(1)$, while if $pT_{\tau}\sum_{l\leq \tau}T_l \rightarrow \infty$ and 
$\sum_{l\leq \tau}T_l<
\left(\beta + \frac{\varepsilon}{2}\right)t_c$ then with probability $1-o(1)$ 
 we have 
\begin{equation}\label{M15}
T_{\tau+1}\leq \frac{10}{1-\beta}  p T_{\tau}\sum_{l\leq \tau}T_l \leq 
\frac{10}{1-\beta}  p\left(\beta + \frac{\varepsilon}{2}\right)^2t_c^2 =O\left( \frac{t_c}{pn}\right).
\end{equation}
 Hence, in either case for $K=\tau+1$ we have \whp
\[
 \sum_{l\leq K}T_l<
\left(\beta + \frac{\varepsilon}{2}\right)t_c+o(t_c)<(\beta +\varepsilon) t_c,\]
for any $\varepsilon >0$. 
This yields

\begin{align}\label{M13}
{\P}\Biggl\{\left( \sum_{l\leq K}T_l<
(\beta + \varepsilon)t_c\right) \cap & \left( K=\tau+1\right)
\mid \sum_{l\leq \tau}T_l<
\left(\beta + \frac{\varepsilon}{2}\right)t_c,  T_1= \beta t_c\Biggr\}
\\ & = {\P}\left\{ K=\tau+1
\mid \sum_{l\leq \tau}T_l<
\left(\beta + \frac{\varepsilon}{2}\right)t_c,  T_1= \beta t_c\right\}-O( h^{-1}  )
\nonumber
\\ & ={\P}\left\{ |D_{\tau+1}|=0
\mid \sum_{l\leq \tau}T_l<
\left(\beta + \frac{\varepsilon}{2}\right)t_c,  T_1= \beta t_c\right\}-o(1)
.
\nonumber
\end{align}
%-------------------------------------------------------------------------------------------------------------------------------------
%-------------------------------------------------------------------------------------------------------------------------------------
%-------------------------------------------------------------------------------------------------------------------------------------

By (\ref{M15})  we have $T_{\tau+1}=O(h)=o(np)$ with probability at least $1-O(  h^{-1}  )$. Then under condition
$\sum_{l\leq \tau}T_l=O(t_c)$ we have
\[pT_{\tau+1}\sum_{l\leq \tau+1}T_l=pO(h(t_c+h)) =o(1) \ \mbox{ and } \  T_{\tau+1}\sum_{l\leq \tau+1}T_l/n=o(1),\]
which by Lemma \ref{CJ2} yields 
\[{\P}\left\{ |D_{\tau+1}|=0
\mid \sum_{l\leq \tau}T_l<
\left(\beta + \frac{\varepsilon}{2}\right)t_c,  T_1= \beta t_c\right\}=1+o(1).
\]
Combining the last bound with (\ref{M13}) and substituting the result into 
(\ref{M12}) we get 
\begin{equation}\label{M16}
{\P}\left\{\sum_{l\leq K}T_l<
(\beta + \varepsilon)t_c \mid T_1= \beta t_c\right\} \geq 
(1-o(1))
{\P}\left\{\sum_{l\leq \tau}T_l<
\left(\beta + \frac{\varepsilon}{2}\right)t_c \mid  T_1= \beta t_c\right\}.
\end{equation}
Consider now
%-------------------------------------------------------------------------------------------------------------------------------------
%-------------------------------------------------------------------------------------------------------------------------------------
%-------------------------------------------------------------------------------------------------------------------------------------
\begin{align}
{\P}& \left\{\sum_{l\leq \tau}T_l<
\left(\beta + \frac{\varepsilon}{2}\right)t_c \mid  T_1= \beta t_c\right\}
\nonumber
\\ & \geq 
{\P}  \left\{
\left(
\sum_{l\leq \tau}T_l<
\left(\beta + \frac{\varepsilon}{2}\right)t_c \right)\cap
\left( T_{k+1} \leq \frac{10}{1-\beta} p T_k 
\sum_{l=1}^{k}T_l, 1\leq k \leq \tau
\right)\mid  T_1= \beta t_c, \tau\right\}
\nonumber
\\
& =
{\P}\left\{T_{k+1} \leq \frac{10}{1-\beta} p T_k 
\sum_{l=1}^{k}T_l, 1\leq k \leq \tau
\mid  T_1= \beta t_c, \tau\right\},
\label{Ll}
\end{align}
where the last equality is due to (\ref{M17}) and the fact that 
\[
\frac{T_1}{1-(c_1/pn)} = \frac{\beta}{1-(c_1/pn)} t_c < \left(\beta + \frac{\varepsilon}{2}\right)t_c
\]
for any fixed $\varepsilon>0$.
Now using the same argument as in  (\ref{M3}) we derive from (\ref{Ll}) 
\[{\P}\left\{\sum_{l\leq \tau}T_l<
\left(\beta + \frac{\varepsilon}{2}\right)t_c \mid  T_1= \beta t_c\right\}
\geq {\E}\left( \left(1-\tau O( h^{-1})\right)
\ind \{\tau<K_0\}  \right)  .
\]
Making use of Proposition \ref{Pfi}  we derive from here
\[
{\P}\left\{\sum_{l\leq \tau}T_l<
\left(\beta + \frac{\varepsilon}{2}\right)t_c \mid  T_1= \beta t_c\right\}
\geq (1-K_0 O( h^{-1}  ))^2.
\]
Substituting the last bound into (\ref{M16}) we finally get
\[
{\P}\left\{\sum_{l\leq K}T_l<
(\beta + \varepsilon)t_c \mid T_1= \beta t_c\right\} \geq 
(1-o(1)) (1-K_0 O( h^{-1}  ))^2=1-o(1),
\]
where the last equation is due to the property (\ref{Ap1}).
This proves statement \ref{T1i} of Theorem \ref{T1}.

\subsubsection{Supercritical case}

Let us turn to the statement \ref{T1ii} of Theorem \ref{T1}.
First, we note that by 
 the corresponding part  $(ii)$ of Theorem 3.1 \cite{JLTV} cited above
 the bootstrap percolation process on $G_{n,p}^1$
accumulates at least 
$n-O_p(pn^2e^{-pn})$ vertices. 
Given that $K= n-O(pn^2e^{-pn})$ are active, the 
 number of remaining vertices which do not have any connection to these 
$K$ active vertices has Bin$(n-K, (1-p)^K)$ distribution.
The expectation of this number is
bounded from above by 
\begin{equation}\label{M20}
(n-K)(1-p)^K \leq  O(pn^2e^{-pn})e^{-pn+O((pn)^2e^{-pn})}
=  O(e^{-2pn+ \log n + \log (pn)}) =o(1),
\end{equation}
where we used the assumption that $2pn \gg \log n+ \log (pn).$
Hence, \whp \  each of the remaining vertex has at least one 
link to the active set. This allows the percolation propagate through 
 the short connections: each vertex on the boundary of active set becomes active if it also has a long connection.
This completes the proof of Theorem \ref{T1}.\hfill$\Box$

\bigskip

{\bf Acknowledgement.} The authors thank the referees for 
 the very detailed reading and remarks which helped to improve the presentation and the proofs.

%%%%%%%%%%%%%%%%%%%%%%%%%%%%%%%%%%%%%%%%%%%%

\end{document}